\documentclass{amsart}
\usepackage{amsfonts,amsmath,amsthm,color, amssymb}
\usepackage[all]{xy}
\usepackage[pagebackref, colorlinks=true, linkcolor=blue, citecolor=blue]{hyperref}

\begin{document}

\newtheorem{theorem}{Theorem}[section]
\newtheorem{lemma}[theorem]{Lemma}
\newtheorem{proposition}[theorem]{Proposition}
\newtheorem{corollary}[theorem]{Corollary}

\newtheorem*{theorems}{Main Theorem}
\newtheorem*{maincor}{Corollary}

\theoremstyle{definition}
\newtheorem{definition}[theorem]{Definition}
\newtheorem{example}[theorem]{Example}

\theoremstyle{remark}
\newtheorem{remark}[theorem]{Remark}

\numberwithin{equation}{section}

\newenvironment{magarray}[1]
{\renewcommand\arraystretch{#1}}
{\renewcommand\arraystretch{1}}

\newcommand{\quot}[2]{
{\lower-.2ex \hbox{$#1$}}{\kern -0.2ex /}
{\kern -0.5ex \lower.6ex\hbox{$#2$}}}

\newcommand{\mapor}[1]{\smash{\mathop{\longrightarrow}\limits^{#1}}}
\newcommand{\mapin}[1]{\smash{\mathop{\hookrightarrow}\limits^{#1}}}
\newcommand{\mapver}[1]{\Big\downarrow
\rlap{$\vcenter{\hbox{$\scriptstyle#1$}}$}}
\newcommand{\liminv}{\smash{\mathop{\lim}\limits_{\leftarrow}\,}}

\newcommand{\Set}{\mathbf{Set}}
\newcommand{\Art}{\mathbf{Art}}
\newcommand{\solose}{\Rightarrow}

\newcommand{\specif}[2]{\left\{#1\,\left|\, #2\right. \,\right\}}
\newcommand{\xmid}{\;\middle|\;}

\renewcommand{\bar}{\overline}
\newcommand{\de}{\partial}
\newcommand{\debar}{{\overline{\partial}}}
\newcommand{\per}{\!\cdot\!}
\newcommand{\Oh}{\mathcal{O}}
\newcommand{\sA}{\mathcal{A}}
\newcommand{\sB}{\mathcal{B}}
\newcommand{\sC}{\mathcal{C}}
\newcommand{\sF}{\mathcal{F}}
\newcommand{\sE}{\mathcal{E}}
\newcommand{\sG}{\mathcal{G}}
\newcommand{\sH}{\mathcal{H}}
\newcommand{\sK}{\mathcal{K}}
\newcommand{\sI}{\mathcal{I}}
\newcommand{\sL}{\mathcal{L}}
\newcommand{\sM}{\mathcal{M}}
\newcommand{\sN}{\mathcal{N}}
\newcommand{\sP}{\mathcal{P}}
\newcommand{\sR}{\mathcal{R}}
\newcommand{\sS}{\mathcal{S}}
\newcommand{\sU}{\mathcal{U}}
\newcommand{\sV}{\mathcal{V}}
\newcommand{\sX}{\mathcal{X}}
\newcommand{\sY}{\mathcal{Y}}

\newcommand{\U}{\mathcal{U}}

\newcommand{\Aut}{\operatorname{Aut}}
\newcommand{\Mor}{\operatorname{Mor}}
\newcommand{\Def}{\operatorname{Def}}
\newcommand{\Hom}{\operatorname{Hom}}
\newcommand{\Hilb}{\operatorname{Hilb}}
\newcommand{\HOM}{\operatorname{\mathcal H}\!\!om}
\newcommand{\DER}{\operatorname{\mathcal D}\!er}
\newcommand{\Spec}{\operatorname{Spec}}
\newcommand{\Der}{\operatorname{Der}}
\newcommand{\Sym}{\operatorname{Sym}}
\newcommand{\Tor}{{\operatorname{Tor}}}
\newcommand{\Ext}{{\operatorname{Ext}}}
\newcommand{\End}{{\operatorname{End}}}
\newcommand{\END}{\operatorname{\mathcal E}\!\!nd}
\newcommand{\Image}{\operatorname{Im}}
\newcommand{\coker}{\operatorname{coker}}
\newcommand{\tot}{\operatorname{tot}}

\newcommand{\somdir}[2]{\hbox{$\mathrel
{\smash{\mathop{\mathop \bigoplus\limits_{#1}}
\limits^{#2}}}$}}
\newcommand{\tensor}[2]{\hbox{$\mathrel
{\smash{\mathop{\mathop \bigotimes\limits_{#1}}
^{#2}}}$}}
\newcommand{\symm}[2]{\hbox{$\mathrel
{\smash{\mathop{\mathop \bigodot\limits_{#1}}
^{#2}}}$}}
\newcommand{\external}[2]{\hbox{$\mathrel
{\smash{\mathop{\mathop \bigwedge\limits_{#1}}
^{\!#2}}}$}}

\renewcommand{\Hat}[1]{\widehat{#1}}
\newcommand{\dual}{^{\vee}}
\newcommand{\desude}[2]{\dfrac{\de #1}{\de #2}}

\newcommand{\A}{\mathbb{A}}
\newcommand{\N}{\mathbb{N}}
\newcommand{\R}{\mathbb{R}}
\newcommand{\Z}{\mathbb{Z}}
\renewcommand{\L}{\mathbb{L}}
\newcommand{\proj}{\mathbb{P}}
\newcommand{\K}{\mathbb{K}\,}
\newcommand{\bi}{\boldsymbol{i}}
\newcommand{\bl}{\boldsymbol{l}}
\renewcommand{\H}{\mathbb{H}}

\newcommand{\MC}{\operatorname{MC}}
\newcommand{\TW}{\operatorname{TW}}

\newcommand{\contr}{{\mspace{1mu}\lrcorner\mspace{1.5mu}}}

\title{Semiregularity and obstructions of complete intersections}
\date{November 9, 2012}

\author{Donatella Iacono}
\address{\newline  Universit\`a degli Studi di Bari,
\newline Dipartimento di Matematica,
\hfill\newline Via E. Orabona 4,
I-70125 Bari, Italy.}
\email{iacono@dm.uniba.it}
\urladdr{\href{http://www.dm.uniba.it/~iacono/}{www.dm.uniba.it/~iacono/}}

%\href{http://arxiv.org/PS_cache/math/pdf/0507/0507286v1.pdf}
%{\texttt{arXiv:math.AG/0507286}}

\author{Marco Manetti}
\address{\newline
Universit\`a degli studi di Roma ``La Sapienza'',\hfill\newline
Dipartimento di Matematica \lq\lq Guido
Castelnuovo\rq\rq,\hfill\newline
P.le Aldo Moro 5,
I-00185 Roma, Italy.}
\email{manetti@mat.uniroma1.it}
\urladdr{\href{http://www.mat.uniroma1.it/people/manetti/}{www.mat.uniroma1.it/people/manetti/}}

\maketitle

\begin{abstract}

We prove that, on a  smooth projective variety over an algebraically closed field of characteristic 0, the  semiregularity map annihilates every obstruction to embedded deformations of a local complete intersection subvariety  with extendable normal bundle. The proof is based on the theory of $L_{\infty}$-algebras and Tamarkin-Tsigan calculus on the de Rham complex of  DG-schemes.

\end{abstract}

\section*{Introduction}

Let $X$ be a smooth algebraic variety, over an algebraically closed field $\K$ of
characteristic 0, and let $Z\subset X$ be
a locally complete intersection closed subvariety of codimension $p$.
Following \cite{bloch}, the \emph{semiregularity map} $\pi\colon H^1(Z,N_{Z|X})\to
 H^{p+1}(X,\Omega^{p-1}_X)$, where $N_{Z|X}$
is the normal bundle of $Z$ in $X$, can be conveniently described by using local cohomology.
In fact, since $Z$ is a locally complete intersection, it is well defined the \emph{canonical}
 cycle class $\{Z\}'\in\Gamma(X,\sH_Z^p(\Omega_X^{p}))$ \cite[p. 59]{bloch} and the
contraction with it gives a morphism of sheaves
\[ N_{Z|X}\xrightarrow{\contr \{Z\}'}\sH_Z^p(\Omega_X^{p-1}).\]
Passing to cohomology, we get a map
\[ H^1(Z,N_{Z|X})\xrightarrow{\contr \{Z\}'}H^1(Z,\sH_Z^p(\Omega_X^{p-1}))=
H^{p+1}_Z(X,\Omega_X^{p-1}),\]
where the last equality follows from the spectral sequence of local cohomology.
Then, the semiregularity map is obtained by taking the composition  with
the natural map $H^{p+1}_Z(X,\Omega_X^{p-1})\to H^{p+1}(X,\Omega_X^{p-1})$.

In \cite{bloch}, using Hodge theory and de Rham cohomology, S. Bloch proved
that if $X$ is projective, then the semiregularity map annihilates certain obstructions to
embedded deformations of $Z$ in $X$. These obstructions contain in particular
the curvilinear ones
and, therefore, if the semiregularity map is injective,
then the Hilbert scheme of subschemes of $X$ is smooth at $Z$.

Unfortunately, Bloch's argument is not sufficient to ensure that the semiregularity map
 annihilates every obstruction to deformations. We have two main reasons to extend Bloch theorem
 to every obstruction: the first is for testing the power of \emph{derived deformation theory} in a
problem where classical deformation theory has failed for almost 40 years. This new approach  already
worked when $Z$ is a smooth submanifold,
see \cite{ManettiSemireg,iaconoUMI,IaconoSemireg} and Remark~\ref{rem.locallytrivialdeformations},
and the solution of this particular
case has given a deep insight  about the most
appropriate formulation and more useful tools of derived deformation theory.
The second reason is related with the theory of reduced Gromov-Witten invariants. Indeed, to
 define the GW invariants, one needs the virtual fundamental class, defined through an
obstruction theory. Whenever the obstruction theory is not carefully chosen, then the
virtual fundamental class is zero and the  standard GW theory is trivial.
A way to overcome this problem, and perform a non trivial Gromov Witten theory
is by using a reduced obstruction theory, obtained by considering the kernel of a suitable map
annihilating obstructions \cite{M-P-T, KoThom}.

The philosophy of derived deformation theory may be summarized in the following way
(see e.g. \cite{EDF}):
 over a  field of characteristic 0, every deformation problem is the classical
truncation of an extended
deformation problem, which is controlled by a differential graded Lie algebra via Maurer-Cartan equation
and gauge equivalence. This differential graded Lie algebra is defined up to quasi-isomorphism
and its first cohomology group
is equal to the Zariski tangent space
of the local moduli space. A morphism of deformation theories is essentially a morphism in the derived category
of differential graded Lie algebras
(with quasi-isomorphisms as weak equivalences); the induced morphism in
cohomology gives, in degrees 1 and 2, the tangent and obstruction map, respectively.

Clearly, a morphism from a deformation theory into an unobstructed deformation theory
provides an obstruction map annihilating every obstruction. Using this basic  principle,
we are able to prove that the semiregularity map annihilates
every obstruction under the following additional assumption:

\textbf{Set-up:}  \emph{$Z$ is closed of codimension $p$ in $X$ and there exists a
Zariski open subset $U\subset X$ and a vector bundle $E\to U$ of rank $p$ such
 that $Z\subset U$ and $Z$ is the zero locus of a section $f\in \Gamma(U,E)$.}

This is obviously satisfied for complete intersections of hypersurfaces,
while for $Z$ of codimension 2 we refer to \cite{OSS} for a discussion
about the validity of the above set-up.

Our main results, proved entirely with deformation theory techniques, are summarized in
the next theorem, where $\H^*$ denotes hypercohomology groups.

\begin{theorems} Let $X$ be a smooth algebraic variety, over an algebraically closed field
 of characteristic 0, and $Z\subset X$  a closed subvariety of codimension $p$ as in
the above set-up. Then, the composition of the semiregularity map and the truncation
\[ H^1(Z,N_{Z|X})\mapor{\pi} H^{p+1}(X,\Omega^{p-1}_X)\mapor{\tau}
\H^{2p}(X,\Omega^0_X\mapor{d}\cdots\mapor{d}\Omega^{p-1}_X)\]
annihilates every obstruction to infinitesimal embedded deformations of $Z$ in $X$.
\end{theorems}

Clearly, if the truncation map $H^{p+1}(X,\Omega^{p-1}_X)\mapor{\tau}
\H^{2p}(X,\Omega^0_X\mapor{d}\cdots\mapor{d}\Omega^{p-1}_X)$
is injective then
 the semiregularity map annihilates every obstruction too.
A sufficient condition for the injectivity of $\tau$ is the  degeneration at level $E_1$
of the Hodge-de Rham spectral sequence; in particular, this is true whenever $X$ is
smooth proper over a field of characteristic 0 \cite{Griffone,DI}.

\begin{maincor}
In the same assumption of the main theorem, if  the Hodge-de Rham spectral sequence of $X$ degenerates at level $E_1$, then
the semiregularity map
\[ H^1(Z,N_{Z|X})\mapor{\pi} H^{p+1}(X,\Omega^{p-1}_X)\]
annihilates every obstruction to infinitesimal embedded deformations of $Z$ in $X$.
\end{maincor}

The set-up assumption is purely technical and there is no reason for the failure of the
theorem for general local complete intersection subvarieties.

The two underlying  ideas used
in the proof of the above theorem are:
\begin{enumerate}

\item to give a purely algebraic proof for smooth submanifolds using the ideas of
\cite{algebraicBTT} and to extend it to the case of DG-schemes, considered in the sense of I. Ciocan-Fontanine and M. Kapranov \cite{KAPR,ciokap};

\item to replace the embedding $Z\subset X$ with  a quasi-isomorphic embedding of smooth DG-schemes.
\end{enumerate}

For the second idea it is necessary to assume the existence of the bundle $E$ and without this assumption the approach of DG-schemes seems insufficient. A possible and natural way to overcome this difficulty is to investigate more deeply the problem in the framework of derived algebraic geometry. After the appearence of the first version of this paper an interesting contribution in this direction has been performed by Pridham \cite{Pri}.

The main theorem, as well as its proof, suggests that, from the point of view of deformation
theory, the semiregularity map $\pi$ is not the correct object to study and should be replaced by  its composition with $\tau$;
this partially explains the past difficulties to relate deformations and
semiregularity and to describe it as a component of the differential  of a morphism
of deformation theories.
The vector spaces   $\H^{2p-1}(X,\Omega^0_X\mapor{d}\cdots\mapor{d}\Omega^{p-1}_X)$
and $\H^{2p}(X,\Omega^0_X\mapor{d}\cdots\mapor{d}\Omega^{p-1}_X)$ may be interpreted
as the tangent and obstruction spaces of the $p$-th intermediate Jacobian
of $X$,  respectively,  and $\tau\pi$ as a component of the differential of the (extended)
infinitesimal  Abel-Jacobi map \cite{bib buchFlenner,Pri,FMAJ}.

The assumption for the base field $\K$ to be of characteristic 0 is essential,
both for the method of the proofs and the validity of the main results.
For simplicity and for giving precise references (especially to \cite{Ar,bloch,Sernesi})  we also assume  that $\K$ is algebraically closed, although this assumption is not strictly necessary.

\medskip
We wish to thank the referee for useful comments and for suggestions improving the presentation of the paper.

\medskip

\subsection{Content of the paper}

The first 3 sections contain a list of algebraic preliminaries: we recall the notion and the
main properties of semifree modules, homotopy fibers of morphisms of differential graded
Lie algebras and the  semicosimplicial Thom-Whitney-Sullivan construction.
In Section 4, we reinterpret some results of Hinich, Fiorenza, Iacono and Martinengo
\cite{hinich,FIM} in terms of descent theory of reduced Deligne groupoids.
In Section 5, we describe two differential graded Lie algebras controlling the
embedded deformations of $Z$ as in the above set-up: they can be seen as spaces
of derived sections of some sheaves of derivation of a DG-scheme.
In Sections 6 and 7, we shall prove that $L_{\infty}$ morphisms, Tamarkin-Tsigan
calculus and Cartan homotopies commutes with Thom-Whitney-Sullivan construction, as well
as the $L_{\infty}$ morphism associated with a Cartan homotopy. In Section 8, we prove that
Tamarkin-Tsigan calculus holds also for the algebraic de Rham complex of a DG-scheme, while, in
Section 9, we prove some base change results about local cohomology of quasi-coherent
DG-sheaves over DG-schemes.
Finally, Sections 10 and 11 are devoted to the proof of the main theorem,
using all the result of the previous sections.

\medskip

\subsection{Notation}

Throughout the paper, we work over an algebraically closed field $\K$ of characteristic zero.
All vector spaces, linear maps, tensor products etc. are intended over $\K$, while every graded
object is considered graded over $\Z$.

A DG-vector space  $V$ is the data of a   $\Z$-graded  vector space,
 $V=\bigoplus_{n\in\Z} V^n$ together with a differential $\delta \colon V \to V $
of degree $+1$. For every homogeneous element $v \in V^i$, we denote by $\overline{v}=i$ its
degree.
As usual, we  use the standard notation for cocycles, coboundaries
and cohomology groups: $Z(V)=\ker \delta, B(V)=\Image \delta$ and $H(V)=Z(V)/B(V)$, respectively.
For any integer $n$ and any DG-vector space  $(V, \delta)$, we define the shifted DG-vector space
 $(V[n],\delta_{V[n]})$, where
$V[n]^i=V^{i+n}$ and $\delta_{V[n]}=(-1)^n\delta$.\par
Given two  DG vector spaces $V$ and $W$, we denote by $\Hom^n_{\K}(V,W)$  the space
 of $\K$-linear map $f:V \to W$, such that $f(V^i) \subset W^{i+n}$, for every $i \in \Z$.
Then $\Hom^*_{\K}(V,W)=\oplus_{n\in\Z}\Hom^n_{\K}(V,W)$ has a natural structure of DG-vector
space with differential $\delta(f)=\delta_W\circ f-(-1)^{\bar{f}}f\circ \delta_V$.

For a given graded  vector space $V$, we will use both symbols $\Sym^n(V)$ and $\bigodot^n V$
for denoting its graded symmetric $n$-th power.
The direct sum of all symmetric powers carries a natural structure of graded algebra and also
a natural structure of graded coalgebra.
For the clarity of exposition, we will adopt the following convention:
\begin{enumerate}

\item $\Sym^*(V)=\bigoplus_{n\ge 0}\Sym^n(V)$ is the graded symmetric algebra generated by
$V$;

\item $\bigodot^*V=\bigoplus_{n\ge 0}\bigodot^nV$ is the graded symmetric coalgebra
generated by $V$;

\item $\overline{\bigodot}^*V=\bigoplus_{n\ge 1}\bigodot^nV$ is the reduced graded
 symmetric coalgebra generated by $V$.
\end{enumerate}

In this paper, we use several kind of cohomology groups/sheaves. Unless otherwise specified,
 we shall denote by:
\begin{itemize}

\item $H^n(X,\sF)$  the $n$-th cohomology group of a sheaf $\sF$ on $X$;

\item $\H^n(X,\sF^*)$  the $n$-th hypercohomology group of a complex of sheaves $\sF^*$;

\item $\sH^n_Z(X,\sF)$ the $n$-th cohomology sheaf of a sheaf $\sF$,
 with support in $Z\subset X$ \cite{GrothendieckLC};

\item $\H^n_Z(X,\sF^*)$  the $n$-th hypercohomology group with support in $Z\subset X$  of a complex of sheaves $\sF^*$;

\item $\sH^n(\sF^*)$ the $n$-th cohomology sheaf of a complex of sheaves $\sF^*$;
\end{itemize}

\bigskip
\section{Review of semifree modules over DG-rings}

By a DG-ring we mean a unitary graded commutative ring
$A=\bigoplus_{n\in\Z} A^n$ endowed with a differential $\delta\colon A\to  A$
of degree $+1$. Graded commutativity means that
$ab=(-1)^{\bar{a}\;\bar{b}}ba$ and the graded Leibniz rule is $\delta(ab)=(\delta a)b+(-1)^{\bar{a}}a(\delta b)$.
Notice that
\[\cdots\mapor{\delta}A^{i-2}\mapor{\delta}A^{i-1}\mapor{\delta}A^i\mapor{\delta}\cdots\]
is a complex of $Z^0(A)$-modules; in particular,  every cohomology group $H^i(A)$ is a
$Z^0(A)$-module. A DG-ring which is also a DG-vector space will be called a  DG-algebra.

Given a DG-ring $A$, a DG-module over $A$ is  a $\Z$-graded $A$-module
$M=\bigoplus_{n\in\Z} M^n$, endowed with a differential $\delta\colon M^i\to M^{i+1}$,
satisfying the conditions:
\begin{enumerate}

\item $am=(-1)^{\bar{a}\;\bar{m}}ma$,

\item $\delta(am)=(\delta a)m+(-1)^{\bar{a}}a(\delta m)$,
\end{enumerate}
for every pair of homogeneous elements $a\in A$ and $m\in M$.
As above, the sequence
\[\cdots\mapor{\delta}M^{i-1}\mapor{\delta}M^{i}\mapor{\delta}M^{i+1}\mapor{\delta}\cdots\]
is a complex of $Z^0(A)$-modules; then, also the cohomology groups $H^i(M)$
 are $Z^0(A)$-modules.

A morphism of DG-modules is an $A$-linear map of degree $0$
commuting with differentials; a quasi-isomorphism is a morphism inducing
 an isomorphism in cohomology.

Given two DG-modules  $M$ and $N$ over  $A$, their tensor product $M \otimes_A N$ is
defined as the quotient of $M \otimes_\Z N$ by the graded submodule generated by all the elements
$ma\otimes n -m \otimes an$, for every $m \in M, n \in N$ and $a \in A$;
notice that, the degree of $n \otimes m$ is $\overline{n}+ \overline{m}$. 
Let $(M, \delta_M)$ and $(N, \delta_N)$ be two DG-modules. We define the graded $\Hom$ as the graded
vector space $\Hom_A^*(M,N)=\oplus_{n \in \Z}\Hom_A^n(M,N)$, where
\[
 \Hom_A^n(M,N)=\{\phi \colon M\to N\
 \mid \phi(M^i)\subset N^{i+n},\; \phi(ma)=\phi(m)a,\; i\in\Z,\; m \in M,\; a \in A\}.
\]
Note that $\Hom_A^*(M,N)$ is a DG-module, with left multiplication $(a\phi)(m)=a \phi(m)$,
for all  $a \in A$ and $m \in M$, and differential $\delta':  \Hom_A^n(M,N) \to   \Hom_A^{n+1}(M,N)
$, given by $\delta' (\phi)= \delta_N \circ \phi-(-1)^n \phi \circ \delta_M$
\cite[Chapter IV.3]{manRENDICONTi}.

In order to simplify the terminology, from now on the term  $A$-modules will mean
a DG-module over a DG-ring $A$.

A homotopy between two morphisms $f,g\colon M\to N$ of $A$-modules
is a morphism  $h\in \Hom_A^{-1}(M,N)$, such that
$\delta_N\circ h+h\circ \delta_M=f-g$.
Homotopic morphisms induce the same morphism in cohomology.

\begin{definition}\label{def.semifree} Let $A$ be a DG-ring,
an $A$-module $F$  is called \emph{semifree} if
$F=\oplus_{i\in I}Am_{i}$, $\bar{m_{i}}\in\Z$ and there exists a
filtration $\emptyset=I(0)\subset I(1)\subset\cdots\subset
I(n)\subset\cdots$ such that
\[I=\cup_{n}I(n),\qquad i\in I(n+1)\solose \delta m_{i}\in \oplus_{i\in I(n)}Am_{i}.\]
\end{definition}

\begin{example} If
$F=\oplus_{i\in I}Am_{i}$ and the set of degrees $\{\bar{m_i}\}$ is bounded from above, then
$F$ is semifree: in fact if $t=\max \{\bar{m_i}\}$ it is sufficient to choose
$I(h)=\{i\mid \bar{m_i}> t-h\}$.
\end{example}

\begin{example}\label{ex.polynomialalgebra} Let $B=A[x_1,\ldots,x_n]$,
with every $x_i$ of  degree $\le 0$, then $B$ is a semifree as $A$-module, for every differential. In fact $B\simeq\oplus_{m\in I}Am$, where $I$ is the set of monomials in the variables $x_i$.
\end{example}

The next results are quite standard and easy to prove (see e.g. \cite{A-F}):
\begin{enumerate}

\item Every quasi-isomorphism of semifree $A$-modules is a homotopy equivalence.

\item If $F$ is semifree and $M\to N$ is a quasi-isomorphism of $A$-modules, then also the induced morphism $M\otimes_A F\to N\otimes_A F$ is a quasi-isomorphism.

\item Every $A$-module admits a semifree resolution, i.e., for every $A$-module $M$ there exists a semifree module $F$ and a surjective quasi-isomorphism $F\to M$.

\end{enumerate}

Given a morphism of DG-ring $A\to B$ and a $B$-module $M$, we define
$\operatorname{Der}_A^*(B,M)=\bigoplus_{n \in \Z}\operatorname{Der}_A^n(B,M)$, where
\[
 \Der_A^n(B,M)=\{\phi \in \Hom^n_{A}(B,M) \
 \mid \ \phi(ab)=\phi(a)b +(-1)^{n\overline{a}}a \phi(b)\}.
\]
As for $\Hom_B^*(M,N)$, there exists a structure of $B$-module on $\operatorname{Der}_A^*(B,M)$,
 with the analogous left multiplication $(a\phi)(b)=a \phi(b)$,
for all  $a,b \in B$, and analogous differential $\delta':  \Der_A^n(B,M) \to  \Der_A^{n+1}(B,M)
$, given by $\delta' (\phi)= \delta_M \circ \phi-(-1)^n \phi \circ \delta_B$.

Given a morphism of DG-rings $A \to B$, we denote by $\Omega_{B/A}$  the module of relative K\"{a}hler
differentials of $B$ over $A$  and by $d\colon B \to \Omega_{B/A}$ the universal derivation, determined by the property that for every $B$-module $M$ the composition with $d$ gives an isomorphism of $B$-modules
\[\bi\colon \Der_A^*(B,M)\mapor{\simeq} \Hom^*_B(\Omega_{B/A}, M),\qquad
\bi_\alpha(db)=\alpha(b)\quad \alpha\in \Der_A^*(B,M), b\in B.\]
The construction of $\Omega_{B/A}$ is essentially the same as in the non graded case,
the differential on $\Omega_{B/A}$ commutes with the differential of $B$ and
$d\in \Der^0_A(B,\Omega_{B/A})$.

For later use, we point out that in the situation of Example~\ref{ex.polynomialalgebra} we have
$\Omega_{B/A}=\oplus B  dx_i$, $\delta(dx_i)=d(\delta x_i)$ and then $\Omega_{B/A}$ is a semifree $B$-module.

\bigskip
\section{Homotopy fiber of a morphism of differential graded Lie algebras}

A \emph{differential graded Lie algebra}  is
the data of a differential graded vector space $(L,d)$ together
with a  bilinear map $[-,-]\colon L\times L\to L$ (called bracket)
of degree 0 such that the following conditions are satisfied:
\begin{enumerate}

\item (graded  skewsymmetry) $[a,b]=-(-1)^{\bar{a}\;\bar{b}}[b,a]$.

\item (graded Jacobi identity) $
[a,[b,c]]=[[a,b],c]+(-1)^{\bar{a}\;\bar{b}}[b,[a,c]]$.

\item (graded Leibniz rule) $d[a,b]=[da,b]+(-1)^{\bar{a}}[a,db]$.
\end{enumerate}

In particular,
the Leibniz rule implies that the bracket of a DGLA
$L$ induces a structure of graded Lie algebra on its cohomology
$H^*(L)=\oplus_iH^i(L)$. Moreover, a DGLA is \emph{abelian} if its bracket is trivial.

A \emph{morphism} of differential graded Lie algebras $\chi\colon L \to M$ is  a linear map
that preserves degrees and commutes with brackets and
differentials.

A \emph{quasi-isomorphism} of DGLAs is a morphism
that induces an isomorphism in cohomology. Two DGLAs $L$ and $M$ are said to be
\emph{quasi-isomorphic}, or \emph{homotopy equivalent}, if they are equivalent under the
equivalence relation generated by: $L\sim M$ if there exists
a quasi-isomorphism $\chi\colon L\to M$.
A differential graded Lie algebra is  \emph{homotopy abelian} if it is quasi-isomorphic to an abelian  DGLA.

The \emph{homotopy fiber} of a morphism of DGLAs $\chi\colon L\to M$ is the differential graded Lie algebra
\[TW(\chi):=\{(l, m(t,dt)) \in L \times M[t,dt] \ \mid \  m(0,0)=0, \, m(1,0)=\chi(l). \}
\]

\begin{lemma}\label{lem.criterion}%\label{prop.criterion}
Let $\chi\colon L\to M$ be a morphism of differential graded Lie algebras such that: \begin{enumerate}

\item $\chi\colon L\to M$ is injective,

\item $\chi\colon H^*(L)\to H^*(M)$ is injective.
\end{enumerate}
Then, the homotopy fiber
%\[TW(\chi)=\{(l, m(t,dt)) \in L \times M[t,dt] \ \mid \  m(0,0)=0, \, m(1,0)=\chi(l) %\}
%\]
is homotopy abelian.
\end{lemma}

\begin{proof} This result has been proved in
\cite{algebraicBTT} using $L_{\infty}$-algebras; here we sketch a more elementary proof.
The same argument used in \cite[Proposition 3.4]{algebraicBTT} shows that that
there exists a direct sum decomposition $M=\chi(L)\oplus V$ as a direct sum of differential graded vector space and, therefore, the mapping cone $L\oplus M[-1]$ of $\chi$ is quasi-isomorphic to $V[-1]$. Consider $V[-1]$ as an abelian differential graded Lie algebras and consider the morphism of DGLAs
\[ f\colon V[-1]\to  TW(\chi),\qquad  f(v)=(0, dt\, v).\]
Since the map
\[ TW(\chi)\to L\oplus M[-1],\qquad (l, p(t)m_1+q(t)dt m_2)\mapsto (l, \int_0^1q(t)dt m_2),\]
is a quasi-isomorphism of complexes, it follows that $f$ is a quasi-isomorphism of DGLAs.
\end{proof}

\begin{remark}\label{rem.quasiisoTWcono} Assume that $\chi\colon L\to M$ is an injective morphism of DGLAs, then
its cokernel $M/\chi(L)$ is a differential graded vector space and the map
\[ TW(\chi)\to (M/\chi(L))[-1],\qquad (l,p(t)m_0+q(t)dt m_1)\mapsto \left(\int_0^1q(t)dt\right) m_1 \pmod{\chi(L)},\]
is a surjective quasi-isomorphism.
\end{remark}

%\begin{example}
%Let $W$ be a differential graded vector space and let $U\subset W$
%be a differential graded subspace. Assume that the induced
%morphism $H^*(U)\to H^*(W)$ is  injective, then the inclusion  of DGLAs
%\[
%\chi\colon\{f\in\Hom^*_{\K}(W,W)\mid f(U)\subset U\}\to \Hom^*_{\K}(W,W)
%\]
%satisfies the hypothesis of Lemma~\ref{lem.criterion} and then
%the DGLA $TW(\chi)$ is homotopy abelian.
%\end{example}

\begin{example}\label{ex.exampleclosedclass}
Let $W$ be a differential graded vector space and  $\gamma\in W$  a cocycle with non trivial cohomology class. Then, the inclusion
\[
\chi\colon\{f\in\Hom^*_{\K}(W,W)\mid f(\gamma)=0\}\to \Hom^*_{\K}(W,W)
\]
satisfies the hypothesis of Lemma~\ref{lem.criterion}. Therefore,
the DGLA $TW(\chi)$ is homotopy abelian. In fact, the morphism of complexes $\K\gamma\to W$ is injective in cohomology and then
by K\"{u}nneth formula the map $\Hom^*_{\K}(W,W)\to \Hom^*_{\K}(\K\gamma,W)$ is surjective in cohomology.
\end{example}

\bigskip
\section{Semicosimplicial Thom-Whitney-Sullivan construction}
\label{sec.scla}

\bigskip

Let $\mathbf{\Delta}_{\operatorname{mon}}$ be  the category whose
objects are the finite ordinal sets $[n]\!=\!\{0,1,\ldots,n\}$,
$n=0,1,\ldots$, and whose morphisms are order-preserving injective
maps among them. Every morphism in
$\mathbf{\Delta}_{\operatorname{mon}}$, different from the
identity, is a finite  composition of \emph{coface} morphisms:
\[
\partial_{k}\colon [i-1]\to [i],
\qquad \partial_{k}(p)=\begin{cases}p&\text{ if }p<k\\
p+1&\text{ if }k\le p\end{cases},\qquad k=0,\dots,i.
\]
The relations about compositions of them are generated by
\[ \partial_{l}\partial_{k}=
\partial_{k+1}\partial_{l}\, ,\qquad\text{for every }l\leq k.\]

According to \cite{EZ,weibel}, a \emph{semicosimplicial} object in
a category $\mathbf{C}$ is a  covariant functor $A^\Delta\colon
\mathbf{\Delta}_{\operatorname{mon}}\to \mathbf{C}$. Equivalently,
a semicosimplicial  object $A^\Delta$ is a diagram in
$\mathbf{C}$:
 \[
\xymatrix{ {A_0}
\ar@<2pt>[r]\ar@<-2pt>[r] & { A_1}
      \ar@<4pt>[r] \ar[r] \ar@<-4pt>[r] & { A_2}
\ar@<6pt>[r] \ar@<2pt>[r] \ar@<-2pt>[r] \ar@<-6pt>[r]&
\cdots ,}
\]
where each $A_i$ is in $\mathbf{C}$, and, for each $i>0$,
there are $i+1$ morphisms
\[
\partial_{k}\colon {A}_{i-1}\to {A}_{i},
\qquad k=0,\dots,i,
\]
such that $\partial_{l}
\partial_{k}=\partial_{k+1} \partial_{l}$, for any $l\leq k$.\par

Given a  semicosimplicial differential
graded vector space
\[V^{\Delta}:\quad
\xymatrix{ {V_0}
\ar@<2pt>[r]\ar@<-2pt>[r] & { V_1}
      \ar@<4pt>[r] \ar[r] \ar@<-4pt>[r] & { V_2}
\ar@<6pt>[r] \ar@<2pt>[r] \ar@<-2pt>[r] \ar@<-6pt>[r]& \cdots ,}\]
the graded vector space $\prod_{n\ge 0}V_n[-n]$ has two
differentials
\[ d=\sum_{n}(-1)^nd_n,\qquad \text{where}\quad d_n
\text{ is the differential of } V_n,
\]
and
\[
\partial=\sum_{i}(-1)^i\partial_i,\qquad \text{where}
\quad \partial_i\text{ are the coface maps}.\]
More explicitly, if
$v\in V^i_n$, then the degree of $v$ is $i+n$ and
\[ d(v)=(-1)^nd_n(v)\in V^{i+1}_n,\qquad
\partial(v)=\partial_0(v)-\partial_1(v)+\cdots+(-1)^{n+1}
\partial_{n+1}(v)\in V_{n+1}^i.\]
Since $d\partial+\partial d=0$, we define $\tot(V^{\Delta})$ as
the  graded vector space  $\prod_{n\ge 0}V_n[-n]$, endowed
with the differential $d+\partial$.

Let  $V^{\Delta}$ be a   semicosimplicial differential graded
vector space and
 $(A_{PL})_n$  the differential graded commutative algebra
of polynomial differential forms on the standard $n$-simplex
$\{(t_0,\ldots,t_n)\in \K^{n+1}\mid \sum t_i=1\}$ \cite{FHT}:
\[ (A_{PL})_n=\frac{\K[t_0,\ldots,t_n,dt_0,\ldots,dt_n]}
{(1-\sum t_i,\sum dt_i)}.\] For every $n,m$ the tensor product
$(A_{PL})_m \otimes V_n $ is a differential graded vector space and
then also $\prod_n (A_{PL})_n \otimes V_n$ is a differential
graded vector space.

Denoting by
\[ \delta^{k}\colon (A_{PL})_n\to
(A_{PL})_{n-1},\quad
\delta^{k}(t_i)=\begin{cases}t_i&\text{ if }0\le i<k\\
0&\text{ if }i=k\\
t_{i-1}&\text{ if }k<i\end{cases},\qquad k=0,\dots,n,
\]
the face maps, for every $0\le k\le n$, there are well-defined
morphisms of   differential graded vector spaces
\[
 \delta^{k} \otimes Id \colon (A_{PL})_{n} \otimes  V_n  \to (A_{PL})_{n-1} \otimes
 V_n,\]
\[Id \otimes \partial_{k} \colon
(A_{PL})_{n-1}   \otimes  V_{n-1}   \to (A_{PL})_{n-1} \otimes  V_{n}.
\]

The Thom-Whitney-Sullivan differential graded vector space of $V^\Delta$
is denoted by ${TW}(V^\Delta)\subset \prod_n(A_{PL})_n \otimes V_n$ and
is the graded subspace whose
elements are the sequences $(x_n)_{n\in\mathbb{N}}$ satisfying the
equations
\[( \delta^{k} \otimes Id)x_n=
(Id \otimes \partial_{k})x_{n-1},\; \text{ for every }\; 0\le k\le
n.
\]

In \cite{whitney}, Whitney noted  that the integration maps
\[ \int_{\Delta^n}\otimes \operatorname{Id}\colon (A_{PL})_{n}\otimes
V_n\to {\mathbb K}[-n]\otimes V_n=V_n[-n]\] give a quasi-isomorphism
of differential graded vector spaces
\[
I\colon ( TW(V^\Delta), d_{TW})\to
({\tot}(V^\Delta),d_{\tot}).
\]
Further details    can be found in
 \cite{navarro,getzler,cone,chenggetzler}.

\begin{example}\label{ex.cech semicosimplicial}
Let $\sL$ be a  sheaf of differential graded vector spaces over an algebraic variety $X$ and  $\sU=\{U_i\}$  an  open cover of $X$; assume that the set of indices $i$ is totally ordered. Then, we can define
the semicosimplicial DG vector space of  \v{C}ech  alternating cochains of $\sL$ with respect to the cover $\sU$:
\[ \sL(\sU):\quad \xymatrix{ {\prod_i\mathcal{L}(U_i)}
\ar@<2pt>[r]\ar@<-2pt>[r] & { \prod_{i<j}\mathcal{L}(U_{ij})}
      \ar@<4pt>[r] \ar[r] \ar@<-4pt>[r] &
      {\prod_{i<j<k}\mathcal{L}(U_{ijk})}
\ar@<6pt>[r] \ar@<2pt>[r] \ar@<-2pt>[r] \ar@<-6pt>[r]& \cdots}.\]
Clearly, in this case, the total complex $\tot(\sL(\sU))$ is the associated \v{C}ech complex $C^*(\sU,\sL)$.
We will denote by $TW(\sU,\sL)$ the associated Thom-Whitney complex. The integration map
$TW(\sU,\sL)\to C^*(\sU,\sL)$ is a surjective quasi-isomorphism.
If $\sL$ is a quasi-coherent DG-sheaf and every $U_i$ is affine, then the cohomology of  $TW(\sU,\sL)$ is the same of the cohomology of $\sL$.
\end{example}

\begin{example}
Let
 \[\mathfrak{g}^\Delta:\quad
\xymatrix{ {{\mathfrak g}_0} \ar@<2pt>[r]\ar@<-2pt>[r] & {
{\mathfrak g}_1}
      \ar@<4pt>[r] \ar[r] \ar@<-4pt>[r] & { {\mathfrak g}_2}
\ar@<6pt>[r] \ar@<2pt>[r] \ar@<-2pt>[r] \ar@<-6pt>[r]& \cdots ,}
\]
be a semicosimplicial differential graded Lie algebra, i.e.,
each ${\mathfrak g}_i$ is a DGLA each $\partial_{k}$ is a morphism of DGLAs.
Then, in this case too, we can apply the
Thom-Whitney construction: it is evident
$ {TW}(\mathfrak{g}^\Delta)$ has a structure of a differential graded lie algebra.
%Moreover,  the morphisms $I,E,h$ are functorial and commute with
%morphisms of semicosimplicial DGLAs.
\end{example}

\begin{example}\label{ex.tot TW morfismo}
 Let $\chi:L \to M$ be a morphism of differential graded Lie algebras.
Then, we can consider the  semicosimplicial DGLA
\[
\xymatrix{  \chi^\Delta \colon \qquad L
\ar@<2pt>[r]\ar@<-2pt>[r] & M
      \ar@<4pt>[r] \ar[r] \ar@<-4pt>[r] & 0
\ar@<6pt>[r] \ar@<2pt>[r] \ar@<-2pt>[r] \ar@<-6pt>[r]& \cdots ,}
\qquad \mbox{with} \quad \partial_0= \chi \ \mbox{and} \ \partial_1=0.
\]
It turns out \cite{ManettiSemireg,cone} that
the total complex   $\tot(\chi^{\Delta})$ coincides  with the
mapping cone of $\chi$, i.e.,
\[ \tot(\chi^{\Delta})^i=L^i\oplus M^{i-1},\qquad d(l,m)=(dl,\chi(l)-dm),\]
and the Thom-Whitney-Sullivan construction
coincides  with the
homotopy fiber of $\chi$:
\[
TW(\chi^\Delta)=TW(\chi)=\{(l, m(t,dt)) \in L \times M[t,dt] \ \mid \  m(0,0)=0, \, m(1,0)=\chi(l)\}.
\]
\end{example}

Whenever $ \mathfrak{g}^\Delta$
is a  semicosimplicial differential graded Lie algebra, we have just noticed that ${TW}(\mathfrak{g}^\Delta)$ is a differential graded Lie algebra.
Moreover, every $\mathfrak{g}_i$ is a DGLA and so, in particular,  a
differential graded vector space, thus we can consider the total
complex $\tot(\mathfrak{g}^\Delta)$.
It turns out that   the complex ${\tot}(\mathfrak{g}^{\Delta})$ has no natural DGLA
structure,  even in the easy case of Example \ref{ex.tot TW morfismo} of a morphism of DGLAs \cite[Example 2.7]{algebraicBTT}.

\begin{lemma}\label{lem.naturalmap}
Let $\mathfrak{g}^\Delta$ be a semicosimplicial DGLA, $L$
a DGLA   and $\varphi\colon L \to {\mathfrak g}_0$ a morphism of DGLAs,
such that $ \partial_{0} \circ \varphi =\partial_{1} \circ \varphi$.
Then it is well defined a morphism of DGLAs $h\colon L \to TW(\mathfrak{g}^\Delta)$  giving a commutative  diagram
\[ \xymatrix{ L\ar[r]^{h\quad}\ar[rd]_{\psi}& TW(\mathfrak{g}^\Delta) \ar[d]^{I}\\
& \tot(\mathfrak{g}^\Delta),}\]
where $\psi\colon L \to \tot(\mathfrak{g}^\Delta)$ is the composition of $\varphi$ with the inclusion
$\mathfrak{g}_0\subset \tot(\mathfrak{g}^\Delta)$.
\end{lemma}

\begin{proof}
A straightforward computation shows that the map
$h\colon L \to TW(\mathfrak{g}^\Delta)$ defined as
\[
h(l)=(1\otimes \varphi(l),1\otimes \partial_{0} (\varphi(l)), 1\otimes \partial_{0}^2
(\varphi(l)), \ldots, 1\otimes \partial_{0}^n (\varphi(l)), \ldots),\]
is a well defined  morphism of differential graded Lie algebras.

Since $I$ contracts the polynomial differential forms in $(A_{PL})_n$ by integrating
over the standard simplex $\Delta_n$,
we have that,  $I(h(l))=\varphi (l) \in {\mathfrak g}_0^i$, for every $l \in L^{i}$.
\end{proof}

\bigskip
\section{Descent of formal reduced Deligne groupoid}

Denote by $\Set$  the category of sets (in a fixed universe) and
by $\mathbf{Grpd}$ the category of small groupoids. We shall consider $\Set$ as a full subcategory of $\mathbf{Grpd}$.
Given a small groupoid $G$ it is convenient, for our purposes,
to think to their objects and arrows  as vertices and edges of its nerve, respectively; therefore, $G_0$ will denote the set of objects and $G_1(x,y)$ the set of morphisms from $x$ to $y$.

Finally, denote by  $\Art$   the category of local Artin
$\K$-algebras with residue field $\K$. Unless otherwise specified,
for every  object   $A\in \mathbf{Art}$, we denote by
$\mathfrak{m}_A$ its maximal ideal.

Let $\mathbf{C}$ be a category with a final object $*$,
throughout this paper, by a formal object in $\mathbf{C}$, we shall mean a covariant functor
$F\colon\Art \to \mathbf{C}$ such that $F(\K)=*$. In particular, a formal   groupoid is a functor $G\colon\Art\to \mathbf{Grpd}$ such that $G(\K)=*$, and  a formal set is a functor $G\colon\Art\to \mathbf{Set}$ such that $G(\K)=*$ (these last ones are also called functors of Artin rings).

\smallskip

\begin{definition}

Let $L$ be a fixed nilpotent differential graded Lie algebra.
The Maurer-Cartan set associated with $L$ is
\[\MC(L)=\left\{x\in L^1\;\middle|\; dx+\dfrac{1}{2}[x,x]=0\right\}.
\]
The gauge action $\ast:\exp(L^0)\times
\MC_L\longrightarrow {\MC}_L$ is defined by the explicit formula
\[
e^a \ast x:=x+\sum_{n\geq 0} \frac{ [a,-]^n}{(n+1)!}([a,x]-da).
\]

Finally, the    \emph{Deligne  groupoid} associated with $L$ is the  groupoid $\operatorname{Del}(L)$ defined as follows:
\begin{enumerate}
\item   the objects are $\operatorname{Del}(L)_0 =\MC(L),$
\item the morphisms are $\operatorname{Del}(L)_1(x,y)= \{e^a\in \exp(L^0) \mid  e^a*x=y\}$,
for $x,y \in \operatorname{Del}(L)_0$.
\end{enumerate}
\end{definition}

For each $x \in \MC (L)  $ the
\emph{irrelevant stabilizer} of $x$ is defined as the subgroup:
\[
I(x)=\{e^{dh+[x,h]}|\, h \in L^{-1}\} \subset
\operatorname{Del}(L)_1(x,x).
\]
Note that, since $e^{dh+[x,h]}*x=x$, the irrelevant stabilizer $I(x)$
  is contained in the stabilizer of $x$ under the gauge action.
Moreover,  $I(x)$ is a normal subgroup of the stabilizer of $x$, since
for any $a \in L^0 $ we have (see e.g.  \cite{Kont94,ManettiSemireg,yek})
\begin{equation}\label{equ.coniugio irrilevanti}
e^a I(x)e^{-a}=I(y),  \qquad \mbox{ with } \qquad
y=e^a*x.
\end{equation}

The above formula implies also that, for every $x,y\in MC(L)$, we have
a natural isomorphism
\[  \frac{\operatorname{Del}(L)_1(x,y)}{I(x)}=\frac{\operatorname{Del}(L)_1(x,y)}{I(y)},\]
with $I(x)$ and $I(y)$ acting in the obvious way.

\begin{definition}[{\cite{Kont94,yek}}]
The  \emph{ reduced  Deligne  groupoid } associated with a nilpotent differential graded Lie algebra $L$ is the  groupoid $\overline{\operatorname{Del}}(L)$ having the same objects as $\operatorname{Del}(L)$ and as morphisms \[\overline{\operatorname{Del}}(L)_1(x,y):=
\frac{\operatorname{Del}(L)_1(x,y)}{I(x)}=\frac{\operatorname{Del}(L)_1(x,y)}{I(y)}.\]

\end{definition}

\begin{example}\label{ex.deligneforembedded}
Let $B,C$ be two $\K$-algebras, $\{x_i\}$ a set of indeterminates of
non positive degree and
$R=B[\{x_i\}]\to C$ a quasi-isomorphism of  DG-algebras and denote by $\delta\in
\Der^1_{B}(R,R)$ the differential of $R$.
Notice that as a graded algebra, $R$ is a free $B$-algebra with generators $x_i$
of degree $\le 0$. Moreover, $R^i=0$ for every $i>0$, $H^0(R,\delta)=C$ and
$H^i(R,\delta)=0$, for every $i\not=0$.
Given a local Artin $\K$-algebra $A$, let us describe the Deligne groupoid
and the irrelevant stabilizers of  the nilpotent
DGLA $L=\Der^*_B(R,R)\otimes\mathfrak{m}_A$.

There is a natural bijection between Maurer-Cartan elements and differentials
$\rho'\colon R\otimes A\to R\otimes A$, which are $B\otimes A$ linear and equal
to $\delta$ modulus $\mathfrak{m}_A$, while the set of morphisms, in the Deligne
groupoid, between $\rho$ and $\rho'$  is the set of $B\otimes A$-linear
isomorphisms $(R\otimes A,\rho)\to (R\otimes A,\rho')$, which are the identity modulus
$\mathfrak{m}_A$.

\begin{lemma}\label{lem.stabilizzatoriirrilevanti}
In the notation above, the irrelevant stabilizer $I(\rho)$ is the group of
 $B\otimes A$-linear automorphisms of the DG-algebra $(R\otimes A,\rho)$, which are
the identity modulus
$\mathfrak{m}_A$ and inducing the identity in cohomology.
\end{lemma}

\begin{proof}
Let $a\in \Der_B^0(R,R)\otimes\mathfrak{m}_A$, then by definition
$e^a\in I(\rho)$ if and only if there exists $b\in \Der_B^{-1}(R,R)\otimes\mathfrak{m}_A$
such that $a=[\rho,b]$. On the other hand, $e^a$ is an isomorphism of $(R\otimes A,\rho)$
if and only if $[\rho,a]=0$,  and induces the identity in cohomology if and only if
$a$ induces the trivial map in cohomology.
Therefore, we only need to prove that if $a\colon R\otimes A\to R\otimes A$ is a morphism of complexes which is trivial in cohomology then  $a=\rho b+b\rho$ for some
$b\in \Der_B^{-1}(R,R)\otimes\mathfrak{m}_A=\Der_B^{-1}(R,R\otimes\mathfrak{m}_A)$.
Since, the derivations $a,b$ are uniquely determined by the values $a(x_i),b(x_i)$, we will solve the
equations $a(x_i)=(\rho b+b\rho)(x_i)$ recursively by induction on $-\deg(x_i)$.
If $x_i$ has degree $0$, then $\rho(x_i)=0$ and then
$a(x_i)\in \rho(R^{-1}\otimes \mathfrak{m}_A)$ (since $a$ is trivial in cohomology): therefore, we can choose
$b(x_i)\in R^{-1}\otimes \mathfrak{m}_A$ such that
$a(x_i)=\rho b(x_i)$.
Next, assume that $x_i$ has degree $k<0$ and that $b(x_j)$ is defined for every $x_j$ of degree bigger than $k$. Then $y_i:=a(x_i)-b(\rho x_i)$ is defined and
$\rho(y_i)=\rho a(x_i)-\rho b (\rho x_i)=a(\rho x_i)-(a-b\rho)(\rho x_i)=0$.
Since $H^k(R\otimes A,\rho)=0$, there exists $b(x_i)\in R^{k-1}\otimes \mathfrak{m}_A$
such that $\rho b(x_i)=y_i$.
\end{proof}

We point out that the fact that the cohomology of $R$ is concentrated in degree 0 plays an  essential role in the above proof: for instance, if we take  $R$ as the Koszul complex of $x^2,x^3\in \K[x]$ then the result of Lemma~\ref{lem.stabilizzatoriirrilevanti} fails to be true.
\end{example}

\noindent  For any DGLA $L$, the previous constructions allow us to define the corresponding formal objects:

\begin{enumerate}
\item the \emph{Maurer-Cartan functor} $\MC_L\colon \mathbf{Art}\to
\mathbf{Set}$ by setting \cite{ManettiSeattle}:
\[
\MC_L(A)=\MC (L\otimes \mathfrak{m}_A); \]
\item the deformation functor $\Def_L:\Art \longrightarrow \Set$:
\[
\Def_L(A)=\frac{\MC_L(A)}{\text{gauge}};\]
\item the \emph{ formal reduced  Deligne  groupoid } $\overline{\operatorname{Del}}_L\colon \Art\to \mathbf{Grpd}$:
\[\overline{\operatorname{Del}}_L(A)=\overline{\operatorname{Del}}(L\otimes \mathfrak{m}_A).\]
\end{enumerate}

Every morphism of DGLAs induces a  morphism of formal reduced groupoids and,
therefore, a morphism of the associated Maurer-Cartan and deformation functors.
A basic result asserts that if
$L$ and $M$ are quasi-isomorphic DGLAs, then the associated
functor $\overline{\operatorname{Del}}_L$ and $\overline{\operatorname{Del}}_M$
are equivalent \cite{GoMil1,Kont94,ManettiSeattle,yek}.

Given a semicosimplicial groupoid
 \[
G^{\Delta}:\qquad \xymatrix{ {{G}_0}
\ar@<2pt>[r]\ar@<-2pt>[r] & { {G}_1}
      \ar@<4pt>[r] \ar[r] \ar@<-4pt>[r] & { {G}_2}
\ar@<6pt>[r] \ar@<2pt>[r] \ar@<-2pt>[r] \ar@<-6pt>[r]&
\cdots}
\]
its total space is the groupoid
$\operatorname{tot}(G^{\Delta})$ defined in the following way \cite{hinich,FMM}:
\begin{enumerate}

\item The objects of $\operatorname{tot}(G^{\Delta})$ are the
pairs $(l,m)$ with $l$ an object in $G_0$ and $m$ a morphism in $G_1$ between $
\partial_{0}l$ and $\partial_{1}l$. Moreover  the three images of $m$ via the maps $\partial_{i}$ are the edges of a 2-simplex in the nerve of $G_2$, i.e.
\[   (\partial_{0}m)(\partial_{1}m)^{-1}(\partial_{2}m)=1 \text{ in }
(G_2)_1(\partial_{2}\partial_{0}l,\partial_{2}\partial_{0}l).\]

\item The morphisms between $(l_0,m_0)$ and $(l_1,m_1)$ are morphisms $a$ in $G_0$
between $l_0$ and $l_1$ making the diagram
\[
\xymatrix{
\partial_{0}l_0\ar[r]^{m_0}\ar[d]_{\partial_{0}a}&\partial_{1}l_0\ar[d]^{{\partial_{1}a}}\\
\partial_{0}l_1\ar[r]^{m_1}&\partial_{1}l_1
}
\]
commutative in $G_1$.
\end{enumerate}

\begin{example}\label{ex.totdiinsiemi}
Let  \[
G^{\Delta}:\qquad \xymatrix{ {{G}_0}
\ar@<2pt>[r]\ar@<-2pt>[r] & { {G}_1}
      \ar@<4pt>[r] \ar[r] \ar@<-4pt>[r] & { {G}_2}
\ar@<6pt>[r] \ar@<2pt>[r] \ar@<-2pt>[r] \ar@<-6pt>[r]&
\cdots}
\]
be a semicosimplicial groupoid. Assume that for every $i$ the natural map $G_i\to \pi_0(G_i)$ is an equivalence, i.e., every $G_i$ is equivalent to a set. Then also
$\operatorname{tot}(G^{\Delta})$ is equivalent to a set and, more precisely, to the equalizer of the diagram of sets
\[
\xymatrix{ {\pi_0({G}_0})
\ar@<2pt>[r]\ar@<-2pt>[r] & { \pi_0({G}_1)}.
}
\]
\end{example}

The next theorem is one of the main results of  \cite{hinich,FIM}.

\begin{theorem}\label{thm.teoremadequeitre}
Let $\mathfrak{g}^{\Delta}$ be a semicosimplicial DGLA, such that $H^{j}(\mathfrak{g}_i)=0$ for every $i$ and $j<0$ and let
 \[
\overline{\operatorname{Del}}_{\mathfrak{g}^{ \Delta}}:\qquad \xymatrix{
{\overline{\operatorname{Del}}_{\mathfrak{g}_0}}
\ar@<2pt>[r]\ar@<-2pt>[r] & { \overline{\operatorname{Del}}_{\mathfrak{g}_1}}
      \ar@<4pt>[r] \ar[r] \ar@<-4pt>[r] & {\overline{\operatorname{Del}}_{\mathfrak{g}_2}}
\ar@<6pt>[r] \ar@<2pt>[r] \ar@<-2pt>[r] \ar@<-6pt>[r]&
\cdots}
\]
the associated semicosimplicial formal reduced Deligne groupoid.
Then, there exists a natural isomorphism of
functors $\Art\to \Set$
\[
\Def_
{\operatorname{TW}(\mathfrak{g}^{ \Delta})} =  \pi_0 (
{\operatorname{tot}}(\overline{\operatorname{Del}}_{\mathfrak{g}^{ \Delta}})).
\]

\end{theorem}

\begin{proof}
This is Theorem 7.6 of \cite{FIM}, expressed in terms of Deligne groupoids and irrelevant stabilizers.
The same result is proved in \cite{hinich} under the assumption that every $\mathfrak{g}_i$ is concentrated in non negative  degrees (and then Deligne=reduced Deligne).
\end{proof}

\section{Embedded deformations of complete intersection}

Let $X$ be a smooth algebraic variety over an algebraically closed field $\K$ of characteristic zero, and $Z\subset X$ a closed subvariety of pure codimension $p$ as in the set-up of the introduction: there exist a Zariski open subset $U\subset X$, a  locally free sheaf  $\sE$, of rank $p$ over $U$, and a section
$f \in \Gamma(U,\sE)$ such that $Z=\{f=0\} \subset U$.

The aim of this section is to describe  two convenient differential graded Lie algebras controlling the
functor $\Hilb_{Z|X}\colon \Art\to \Set$ of infinitesimal embedded deformations of $Z$ in $X$.

The first DGLA is simpler and it has a clear geometric interpretation. The second, which is quasi-isomorphic to the first one, will be used
in the proof of our main result and it is similar to the ones considered in  \cite{ManettiSemireg,donarendiconti} in the case $Z$ smooth.

\subsection{Local case.}

Assume that $U=\Spec P$ is a smooth affine over $\K$ and $\sE=\Oh_U^p$. If $f_1,\ldots,f_p$ are the components of the section $f$, then  the ideal of
$Z$ is $J=(f_1, \ldots , f_p)\subset P$.

It is well known (see e.g. \cite{Ar,Sernesi}) that the set of embedded deformations of $Z$ over a local Artin ring $A$ corresponds naturally to the set of ideals $\tilde{J}\subset P\otimes A$ generated by liftings of $f_1,\ldots,f_p$. Notice that every lifting of $f_i$ may be written as $f_i+g_i$ with $g_i\in P\otimes\mathfrak{m}_A$.

Let $R$ be the Koszul complex of the sequence $f_1,\ldots,f_p$, considered as a DG-algebra;
in other words, $R$ is the polynomial algebra $P[y_1,\dots,y_p]$, where $\deg(y_i)=-1$ and $d(y_i)=f_i$.
In particular, $R^{0}=P$, $R ^{-1}=\bigoplus_{i} P y_i$ and $R^j=0$ for every $j>0$. Since $Z$ is a complete intersection, the natural map $R\to P/J$ is a quasi-isomorphism of DG-algebras.

\begin{lemma}\label{lem. DGLA local embed deform}
In the notation above, the differential graded Lie algebra  $L=\Der^*_{P}(R,R)$ controls the
embedded deformations of $Z \subset X$. More precisely,
there exists an equivalence of formal groupoids
\[ \overline{\operatorname{Del}}_L\to \Hilb_{Z|X}\]
and, therefore, an isomorphism of functors of Artin rings
\[ \Def_L\to \Hilb_{Z|X}.\]
\end{lemma}

\begin{proof}
Let $A$ be a local Artin ring. We have
$\operatorname{MC}_{\Der^*_{P}(R,R)}(A)=\Der^1_{P}(R,R)
\otimes\mathfrak{m}_A=(P\otimes\mathfrak{m}_A)^p$, since every derivation $\eta$ of $R$ of degree 1 is uniquely determined by the sequence $\eta(y_1),\ldots,\eta(y_p)\in P\otimes\mathfrak{m}_A$.
Then Example~\ref{ex.deligneforembedded} gives
a morphism of groupoids
\[ \overline{\operatorname{Del}}_L(A)\to \Hilb_{Z|X}(A)\]
which is surjective on objects.

Let $\eta,\mu\in \Der^1_{P}(R,R)
\otimes\mathfrak{m}_A$ be two derivations giving the same deformation, i.e.,
\[ \tilde{J}=(f_1+\eta(y_1),\ldots,f_p+\eta(y_p))=(f_1+\mu(y_1),\ldots,f_p+\mu(y_p))\]
 as ideal of $P\otimes A$.
Then, by the  flatness of $\tilde{J}$, we have
\[\eta(y_i)-\mu(y_i)\in \ker(\tilde{J}\to \tilde{J}\otimes_A\K=J)=\mathfrak{m}_A\tilde{J}\]
and then there exist $a_{i1},\ldots,a_{ip}\in P\otimes \mathfrak{m}_A$ such that
\[  \eta(y_i)-\mu(y_i)=\sum_{j}a_{ij}(f_j+\mu(y_j)).\]
Taking $c\in \Der^0_P(R,R)\otimes\mathfrak{m}_A$ such that
$e^c(y_i)=y_i+\sum_{j}a_{ij}y_j$ we have that $e^c$ is an
isomorphism of the two DG-algebras $(R\otimes A,d+\eta)$ and $(R\otimes A,d+\mu)$, i.e., $e^c\ast \eta=\mu$ and $\eta,\mu$ are gauge equivalent.
The last step of the proof is exactly Lemma~\ref{lem.stabilizzatoriirrilevanti}.
\end{proof}

\medskip

As regard the second DGLA, consider the DG-algebra
\[ S=P[z_1,\ldots,z_p,y_1,\ldots,y_p],\]
where $\deg{z_i}=0$, $\deg{y_i}=-1$ and $d(y_i)=f_i-z_i$.
Notice that the map
\begin{equation}\label{equ.surjectyivequasiiso}
S\to P,\qquad y_i\mapsto 0,\quad z_i\mapsto f_i
\end{equation}
is a surjective quasi-isomorphism of DG-algebras.

Let $I\subset S$ be the ideal generated by $z_1,\ldots,z_p$, i.e., the
kernel of the projection map $\pi\colon S\to R$.
Then, consider the DGLA  $H{=}\{\eta\in\Der^*_P(S,S)\mid \eta(I)\subset I\}$ and the surjective morphism of DGLAs
\[ \Phi\colon H\to \Der^*_P(R,R),\qquad
\Phi(\eta)(y_i)=\pi(\eta(y_i)).\]

\begin{lemma}\label{lem. sostituzione DGLa con cono}
Let $\chi\colon H\to \Der^*_P(S,S)$ be the inclusion. Then, the horizontal maps
\[ \xymatrix{H\ar[d]_\chi\ar[r]^{\Phi\quad}&\Der^*_P(R,R)\ar[d]\\
\Der^*_P(S,S)\ar[r]&0}\]
induce a quasi-isomorphism of homotopy fibers.
\end{lemma}

\begin{proof} The horizontal maps are surjective and then it is sufficient to prove that
\[ \Der^*_P(S,S)=\Hom^*_S(\Omega_{S/P},S),\qquad
\ker\Phi=\Der^*_P(S,I)=\Hom^*_S(\Omega_{S/P},I)\]
are acyclic; both equalities follow from the fact that
the $S$-module $\Omega_{S/P}$ is semifree and acyclic.
\end{proof}

\subsection{Global case}\label{subsec.global}
The section $f\in \Gamma(U,\sE)$ gives a Koszul-Tate resolution of the structure sheaf of $Z$:
\[
 0 \to \bigwedge^p \sE^\vee \to \bigwedge^{p-1} \sE^\vee \to \cdots \to \bigwedge^2 \sE^\vee
 \to\sE^\vee \to \Oh_U\to \Oh_Z\to 0
\]
and, therefore, a quasi-isomorphism of sheaves of DG-algebras over $\Oh_U$
\[ \sR=\Sym^*_{\Oh_U}(\sE^\vee[1])\to \Oh_Z,\]
where the differential on $\sR$ is given by extending $f^*\colon \sE^\vee\to \Oh_U$ via Leibniz rule.

Let $\DER^*_{\Oh_U}(\sR,\sR)$ be the DG-sheaf of $\Oh_U$-linear derivations of $\sR$, it is a sheaf of DGLAs and consider an affine open covering $\sU=\{U_i\}$, which is trivializing for the sheaf $\sE$.

\begin{theorem}\label{teo. dgla globale HILBERT scheme}
In the above set-up, assume that $\sU$ is an affine open cover of $X$ which is trivializing for the locally free sheaf $\sE$, then the infinitesimal embedded deformations of $Z$ in $X$ are controlled  by the  differential graded Lie algebra
$TW(\sU,\DER^*_{\Oh_U}(\sR,\sR))$.
\end{theorem}

\begin{proof}

By the computations in the local case, for every $i$ the DGLA
$\Gamma(U_i, \DER^*_{\Oh_U}(\sR,\sR))$ controls the embedded deformations
of $Z\cap U_i$ inside $U_i$. A global embedded deformation of $Z$ in $X$ is simply given by
a sequence of embedded deformations of $Z\cap U_i$ inside $U_i$, with isomorphic
restrictions on double intersections $U_{ij}$.
Therefore, to conclude the proof, it is enough  to apply
Theorem~\ref{thm.teoremadequeitre} and Example~\ref{ex.totdiinsiemi}.
\end{proof}

In order to globalize the second local construction, let us consider the graded locally free $\Oh_U$-module $\sE^\vee[1]\oplus \sE^\vee$.
On the graded symmetric $\Oh_U$-algebra $\Sym^*_{\Oh_U}(\sE^\vee[1]\oplus \sE^\vee)$, we consider the unique differential induced by the map of degree +1
\begin{equation}\label{equ.differenzialesuesse}
\sE^\vee[1]\oplus \sE^\vee\to \Sym^*_{\Oh_U}(\sE^\vee[1]\oplus \sE^\vee),\qquad \sE^\vee[1]\ni x\mapsto f^*(x)-x\in \Oh_U\oplus \sE^\vee.
\end{equation}

Let $\sM=\DER_{\Oh_U}^*(\Sym^*_{\Oh_U}(\sE^\vee[1]\oplus \sE^\vee),\Sym^*_{\Oh_U}(\sE^\vee[1]\oplus \sE^\vee))$  and
$\sM_{\perp}\subset \sM$  the subsheaf of derivations preserving the ideal generated by
$\sE^\vee$. Then, the homotopy fiber of the inclusion $\sM_\perp\to \sM$ is
quasi-isomorphic to $\DER^*_{\Oh_U}(\sR,\sR)$ and  we have  the following result.

\begin{theorem}\label{thm.homotopyfibercontrolligHilb}
In the above set-up, assume that $\sU$ is an affine open cover of $U$, then the infinitesimal embedded deformations of $Z$ in $X$ are controlled by the homotopy fiber of the inclusion  of DGLAs
\[ TW(\sU,\sM_\perp)\;\mapor{\chi}\; TW(\sU,\sM).\]
\end{theorem}

\begin{proof}
This is the global version of Lemma~\ref{lem. sostituzione DGLa con cono}, as Theorem~\ref{teo. dgla globale HILBERT scheme} is the global version of Lemma \ref{lem. DGLA local embed deform}.
\end{proof}

\medskip
\subsection{Relation with DG-schemes}\label{subsec.relationDGschemes}

According to \cite{ciokap}, a \emph{DG-scheme} is a pair
$(T, \sR_T)$, where $T$ is an ordinary scheme
and $\sR_T$ is a sheaf of ($\mathbb{Z}_{\le 0}$-graded)
commutative DG-algebras on $T$, such that $\sR_T^0=\Oh_{T}$
and each $\sR_T^i$ is quasi-coherent over $\Oh_T$.
A \emph{morphism  of DG-schemes} is just a morphism of DG-ringed
spaces.

A \emph{closed embedding of DG-schemes} is a morphism  $f\colon (Y,\sR_Y)\to (T,\sR_T)$
such that $f\colon Y\to T$ is a closed embedding of schemes and the induced map
$\sR_T\to f_*\sR_Y$ is surjective.

Any ordinary scheme can be
be considered as a DG-scheme with trivial grading and differential; any ordinary closed subscheme of $T$ can be considered as a closed DG-subscheme of $(T,\sR_T)$.

For any DG-scheme $(T, \sR_T)$, its differential  $\delta\colon \sR_T^i\to\sR_T^{i+1}$
is $\Oh_{T}$-linear and hence $\sH^i (\sR_T)$
are quasi-coherent sheaves on $T$. We define \emph{the degree 0 truncation}
$\pi_0(T, \sR_T)$ as the ordinary closed subscheme of $T$ defined by the ideal
$\delta(\sR_T^{-1})$, and then
with structure sheaf $\sH^0(\sR_T)$.

A \emph{quasi-isomorphism of DG-schemes} is a morphism  $f\colon (Y,\sR_Y)\to (T,\sR_T)$
such that the induced map $\pi_0(Y,\sR_Y)\to \pi_0(T,\sR_T)$ is a isomorphism of schemes and
$\sH^*(\sR_T)\to f_*\sH^*(\sR_Y)$ is an isomorphism of graded sheaves.

It is useful  to see the constructions of Subsection~\ref{subsec.global} in the framework of DG-schemes; more precisely, we will describe two DG-schemes which are quasi-isomorphic to $Z$ and $U$, respectively.

The first DG-scheme is just the pair $(U,\sR)$, defined in the previous subsection; in this case, we have
$\pi_0(U,\sR)=Z$ and the closed embedding $(Z, \Oh_Z) \to (U,\sR)$ is a quasi-isomorphism of DG-schemes.

The other DG-scheme is $(E,\sS)$, where $\pi\colon E\to U$ is the total space of the vector bundle associated with $\sE$ (i.e., $E=\Spec (\Sym^*_{\Oh_U}(\sE^\vee)$) and
$\sS=\Sym^*_{\Oh_E}(\pi^*\sE^\vee[1])$, with the differential on $\sS$  induced by the
morphism as in Equation~\eqref{equ.differenzialesuesse}. Notice that
$\pi_*\sS=\Sym^*_{\Oh_U}(\sE^\vee[1]\oplus \sE^\vee)$ and, by \eqref{equ.surjectyivequasiiso}, the section $f\colon U\to E$ gives a closed embedding and a quasi-isomorphism $(U,\Oh_U)\to (E,\sS)$, while  the zero section $0\colon U\to E$ induces a closed embedding
$(U,\sR)\to (E,\sS)$.

Therefore, we have the following commutative diagram of closed embeddings of DG-schemes
\[
\xymatrix{
(Z, \Oh_Z)\ar[d]\ar[r]^{\;qiso\;}&(U,\sR)\ar[d]   \\
(U,\Oh_U) \ar[r]^{\;qiso\;}&(E,\sS) }
\]

Notice that we have  the following natural isomorphisms of DG-sheaves
\[\sM=\DER^*_{\Oh_U}(\pi_*\sS,\pi_*\sS)=\pi_*\DER^*_{\Oh_U}(\sS,\sS),
\qquad\sM_{\perp}=\pi_*\{\phi\in \DER^*_{\Oh_U}(\sS,\sS)\mid \phi(\sI)\subset \sI\},\]
where $\sI$ is the  ideal sheaf of the closed embedding
$(U,\sR)\to (E,\sS)$.

\bigskip
\section{$L_\infty$ morphisms and obstructions}

We briefly recall the notion of  $L_\infty$-algebras and  $L_\infty$ morphisms.
For a more detailed
description of such structures we refer to
\cite{SchSta,LadaStas,LadaMarkl,EDF,fukaya,K,getzler,CCK,cone} and
\cite[Chapter~IX]{manRENDICONTi}.

\begin{definition}
An $L_{\infty}$ structure  on a graded vector space $V$ is a
sequence $\{q_k\}_{k \geq 1}$ of linear maps $q_k \in \Hom^1_{\K}(\bigodot^k
(V[1]),V[1])$ such that the coderivation
\[Q:\overline{\bigodot}^*(V[1]) \to \overline{\bigodot}^*(V[1]),\]
defined as
\[Q(v_1 \odot \cdots \odot  v_n)=\sum_{k=1}^n \sum_{\sigma \in
S(k,n-k)} \epsilon(\sigma)q_k( v_{\sigma(1)} \odot \cdots \odot
v_{\sigma(k)})\odot v_{\sigma(k+1)} \odot \cdots \odot
v_{\sigma(n)},\]
satisfies $QQ=0$, i.e., if $Q$ is a codifferential on the reduced symmetric graded coalgebra;
here $\epsilon(\sigma)$ denotes the Koszul sign and $S(k,n-k)$ is the set of \emph{unshuffles} of type $(k,n-k)$, i.e.,  the set
of permutations $\sigma$ such that
$\sigma(1)<\sigma(2) < \cdots < \sigma(k)$ and $\sigma(k+1)
<\sigma(k+2) < \cdots < \sigma(n)$.

An $L_{\infty}$-algebra $(V,q_1,q_2,q_3, \ldots )$ is the data of a graded vector space $V$ and an $L_{\infty}$ structure $\{q_i\}$ on it.
\end{definition}

Notice that if $(V,q_1,q_2,q_3, \ldots )$  is an $L_{\infty}$-algebra, then
$q_1q_1=0$ and therefore $(V[1],q_1)$ is a DG-vector space.

\begin{example}[Quillen construction, \cite{Qui}]\label{oss DGLA is L infinito}
Let $(L,d,[\, , \, ])$ be a differential graded Lie algebra and define:
\[
q_1=-d: L[1] \to L[1],\]
\[q_2 \in \Hom^1_{\K}({\bigodot}^2 (L[1]),L[1]),\qquad
q_2(v\odot w)=(-1)^{\bar{v}}[v,w],
\]
and $q_k=0$ for every $k \geq 3$. Then, $(L,q_1,q_2,0,\ldots)$ is
an $L_{\infty}$-algebra.
Explicitly, in this case, the differential $Q$ is given by

$$
Q(v_1 \odot \cdots \odot  v_n)=  \!\! \sum_{\sigma \in
S(1,n-1)} \!\! \epsilon(\sigma)d( v_{\sigma(1)})  \odot \cdots \odot
v_{\sigma(n)}
+   \!\! \sum_{\sigma \in
S(2,n-2)}  \!\! \epsilon(\sigma)q_2( v_{\sigma(1)} \odot
v_{\sigma(2)})\odot v_{\sigma(3)} \odot \cdots \odot
v_{\sigma(n)}
$$

\end{example}

There exist two equally important notions of morphisms of $L_{\infty}$ structures: the linear morphisms (graphically denoted with a solid arrow) and the
$L_{\infty}$ morphisms (denoted with a dashed arrow).

A \emph{linear morphism} $f\colon (V,q_1,q_2, \ldots
)\rightarrow (W,p_1,p_2, \ldots )$ of $L_{\infty}$-algebras is a
morphism of graded vector spaces $f\colon V[1]\to W[1]$ such that
$f\circ q_n=p_n\circ (\odot^n f)$, for every $n>0$.

An \emph{$L_{\infty}$ morphism} $f_{\infty}\colon (V,q_1,q_2, \ldots
)\dashrightarrow (W,p_1,p_2, \ldots )$ of $L_{\infty}$-algebras is the data of a
sequence of morphisms
\[
f_n\in \Hom_{\K}^0({\bigodot}^n (V[1]), W[1]), \qquad n\geq 1,
\]
such that the unique morphism of graded coalgebras
\[
F\colon \overline{\bigodot}^* (V[1]) \to \overline{\bigodot}^* W[1]
\]
lifting $\sum_n f_n :\overline{\bigodot}^* (V[1])  \to
W[1]$, commutes with the codifferentials.
This condition implies that the \emph{linear part} $f_1\colon V[1]\to W[1] $ of an
$L_{\infty}$ morphism $f_{\infty}\colon (V,q_1,q_2, \ldots )\dashrightarrow
(W,p_1,p_2, \ldots )$ satisfies the condition $f_1 \circ q_1
=p_1 \circ f_1$,  and therefore $f_1\colon (V[1],q_1) \to (W[1], p_1)$ is a morphism of
DG-vector spaces.

\begin{remark} Every linear morphism of $L_{\infty}$-algebras is also an $L_{\infty}$ morphism, conversely
an $L_{\infty}$ morphism $f_{\infty}=\{f_n\}$ is  linear
if and only if $f_n=0$, for every $n \geq 2$.
An $L_{\infty}$ morphism between two DGLAs is linear if and only if
it is a morphism of differential graded Lie algebras.
\end{remark}

\begin{lemma} \label{lemma.basechange}
Let $L$ and $M$ be differential graded Lie algebras,  $f_{\infty}\colon L \dashrightarrow M$ an  $L_\infty$ morphism between $L$ and $M$ and $A$ a commutative DG-algebra.  Then,
$f_{\infty}$ induces canonically an $L_\infty$ morphism
\[{(f_A)} _{\infty}\colon L\otimes A \dashrightarrow M \otimes A.\]

\end{lemma}

\begin{proof}
By definition, $f_{\infty}\colon L \dashrightarrow M$ corresponds to a
sequence $\{f_n\}_n$ of
degree zero linear maps, with
$f_n: \bigodot^n  (L[1]) \to M[1]$. Then, for any   associative
graded commutative algebra  $A$,  define the natural extensions
\[ {(f_A)}_n\colon  {\bigodot}^n ((L\otimes A)[1] ) \to (M\otimes A)[1]\]
 as the composition
\[  {\bigodot}^n ((L\otimes A)[1]) \mapor{\mu}
  {\bigodot}^n  (L  [1]) \otimes A  \xrightarrow{f_n\otimes Id_A} M[1]\otimes A
\simeq (M\otimes A)[1],\]%
where $\mu$ is induced by  multiplication on $A$ and by the Koszul rule of signs. The proof that
this induces an $L_\infty$ morphism is completely straightforward.
\end{proof}

 \begin{proposition}\label{prop. morfismo L-infinito limit}
Let $\mathbf{D}$ be a small category, $\mathbf{DGLA}$ the category of differential graded Lie algebras and  $H, G\colon  \mathbf{ D} \to \mathbf{DGLA}$
two functors; denote by $\lim (H)$ and $\lim ( G)$ their limits.
Then, every $L_{\infty}$ natural transformation $H\dashrightarrow G$ induces
an $L_\infty$ morphism
\[f_\infty\colon\lim (H) \dashrightarrow
\lim(G).\]
\end{proposition}

\begin{proof}
An $L_{\infty}$ natural transformation is just a sequence
\[(f_d)_\infty \colon
  H(d)    \dashrightarrow  G(d),\qquad d\in D,\]
of $L_\infty$ morphisms, such that for every morphism $\gamma\colon d\to d'$ in $D$ the
diagram
\[ \xymatrix{H(d)\ar@{-->}[r]^{(f_d)_\infty}\ar[d]^{H(\gamma)}&G(d)\ar[d]^{G(\gamma)}\\
H(d')\ar@{-->}[r]^{(f_{d'})_\infty}&G(d')}\]
is commutative in the category of $L_{\infty}$-algebras. Then,  the conclusion follows immediately from the fact that the Quillen's construction, considered as a functor from the category of DGLAs to the category of locally nilpotent differential graded coalgebras, commutes with limits.
Notice that the forgetful functor from coalgebras to vector spaces does not
commute with limits.
\end{proof}

\begin{corollary}\label{cor. morfismo infinito tra tot}
Let $\mathfrak{g}^\Delta$ and $\mathfrak{h}^\Delta$ be two semicosimplicial DGLAs.
Then every semicosimplicial $L_\infty$ morphism
$\mathfrak{g}^\Delta \dashrightarrow \mathfrak{h}^\Delta$ induces an $L_{\infty}$ morphism
\[{TW}({\mathfrak
g}^\Delta)\dashrightarrow  {TW}({\mathfrak
h}^\Delta). \]
\end{corollary}

\begin{proof}
By definition,  a semicosimplicial $L_{\infty}$ morphism is a sequence of $L_\infty$ morphisms
$\{(f_d)_\infty:  \mathfrak{g}_d  \dashrightarrow \mathfrak{h}_d  \}$, $d\ge 0$, commuting
 with coface maps. By definition, we have
\[
 TW (\mathfrak{g}^\Delta) = \{ (x_n) \in \prod_n   (A_{PL})_n \otimes \mathfrak{g}_n
\mid ( \delta^{k} \otimes Id)x_n=
(Id \otimes \partial_{k})x_{n-1}\;\forall n,k\}.
\]

Then, applying Lemma \ref{lemma.basechange}, every $(f_n)_\infty$ induces
$L_\infty$ morphisms $   (A_{PL})_m \otimes \mathfrak{g}_n   \dashrightarrow (A_{PL})_m
\otimes \mathfrak{h}_n  $.  Finally, it is enough to apply the previous
Proposition \ref{prop. morfismo L-infinito limit} interpreting TW as
an end, and then as a limit.
\end{proof}

\begin{example}\label{ex.quadrato.morfismo.infinito}
Let $\chi\colon L\to M$ and $\eta\colon L'\to M'$ be morphisms of
differential graded Lie algebras. Then, every
commutative diagram
\[ \begin{array}{ccc}
 L &  \stackrel{f}{\dashrightarrow}&L' \\
\mapver{\chi}&&\mapver{\eta}\\
M & \stackrel{f'}{\dashrightarrow}&M' \end{array}\]
with $f$ and $f'$ $L_\infty$ morphisms,
induces   an $L_\infty$ morphism ${TW}(\chi) \dashrightarrow {TW}(\eta)$.

\end{example}

\begin{remark}\label{rem.morphismidecalati}
Via the standard d\'ecalage isomorphisms $\bigodot^n(V[1])\simeq (\bigwedge^n V)[n]$,
every $L_{\infty}$
morphism  $f_\infty \colon V \dashrightarrow W$, associated with a sequence of maps
$f_n:\bigodot^n (V[1] ) \to W[1]$ can be described by a sequence of linear maps
\[ g_n\colon \bigwedge^n V\to W,\qquad n\ge 1,\quad \deg(g_n)=1-n.\]
\end{remark}

Every $L_{\infty}$-morphism  $f_{\infty}\colon L \dashrightarrow M$ of DGLAs induces morphisms
of the associated Maurer-Cartan and deformation functors
\[ f_\infty\colon \MC_L\to \MC_M,\qquad f_\infty\colon
\Def_L \to \Def_M.\]
Moreover, the linear part $f_1\colon L[1]\to M[1]$ induces a morphism in
cohomology, which is compatible with the obstruction maps of $\Def_L$ and $\Def_M$
in the following sense:
given a DGLA $L$, a small extension in the category $\Art$
\[e:\quad 0\to \K \to B \to A\to 0,\qquad A,B\in \Art,\]
and an element $x\in \MC_L(A)$ we can take a lifting $\tilde{x}\in L^1\otimes \mathfrak{m}_B$
and consider the element $h=d\tilde{x}+[\tilde{x},\tilde{x}]/2\in L^2\otimes\K=L^2$.
It is very easy to show that $dh=0$ and that the cohomology class $[h]\in H^2(L)$
does not depend on the choice of the lifting; therefore, it gives the \emph{obstruction map}
$ob_{e}\colon \MC_L(A)\to H^2(L)$.
It is not difficult to prove (see e.g. \cite{FantMan,ManettiSeattle})
that the obstruction map is gauge invariant, thus giving a map
$ob_{e}\colon \Def_L(A)\to H^2(L)$ such that an element $x\in \Def_L(A)$ lifts to
$\Def_L(B)$ if and only if $ob_e(x)=0$. The obstructions of the functor $\Def_L$
are the elements of $H^2(L)$ lying in the image of some obstruction map; for instance,
 if $L$ is abelian then every obstruction is equal to $0$.
The obstruction maps commute with morphisms of DGLAs and quasi-isomorphic DGLAs
have isomorphic associated deformation functors and the same obstructions. This implies that the set of obstructions is a homotopical invariant of the DGLA.

Given an  $L_{\infty}$ morphism  $L \dashrightarrow M$ of DGLAs,  the
map $H^2(L) \to H^2(M)$, induced by its linear part, commutes with the induced morphism $\Def_L \to \Def_M$ and
obstruction maps \cite{EDF}.
In particular, if $M$ is homotopy abelian, then
the obstructions of the functor $\Def_L$ are contained in the kernel of the induced map
$H^2(L) \to H^2(M)$.

The obstruction maps  exist for every  functor of Artin rings
$F \colon \Art \to\Set$ describing infinitesimal deformation of algebro-geometric structures
and one of the main results of \cite{FantMan} is the proof that obstruction maps exist for every
$F$ as above satisfying some mild Schlessinger type conditions.

\begin{example}
Let $\chi\colon L \to M$ be a morphism of DGLAs, then the obstructions of the functor
$\Def_{TW(\chi)}\colon \Art \to \Set$ are contained in
$H^2(TW(\chi))$. If $\chi$ is injective, then $H^2(TW(\chi))\cong H^1(\coker(\chi))$
(see Remark~\ref{rem.quasiisoTWcono}). Moreover, if $\chi$ is also injective in cohomology
then  $TW(\chi)$ is homotopy abelian (Lemma~\ref{lem.criterion}) and then $\Def_{TW(\chi)}$ has only trivial obstructions.
\end{example}

\begin{example}
Let $X$ be a smooth algebraic variety, over an algebraically
 closed field $\K$ of characteristic 0, and
$Z\subset X$  a locally complete intersection closed subvariety.
Let $\Hilb_{Z|X}\colon \Art \to\Set$ the deformation functor of
infinitesimal embedded deformations of $Z$ in $X$.
Then, it is well known that the obstructions are naturally contained in the
 cohomology vector space  $H^1(Z,N_{Z|X})$ of the normal sheaf of $Z$ in $X$, see e.g
\cite[Prop. 3.2.6]{Sernesi}. This is also recovered in this paper as a consequence of
Theorem~\ref{teo. dgla globale HILBERT scheme}.

\end{example}

\bigskip
\section{Cartan homotopies and $L_{\infty}$ morphisms}

Next, let us recall the notion of Cartan homotopy \cite{cone,Periods}.

\begin{definition}\label{def.cartanhomotopy}
Let $L$ and $M$ be two differential graded Lie algebras. A linear map of degree $-1$
\[ \bi\colon L \to M  \]
is called a \emph{Cartan homotopy} if, for every $a,b\in L$, we
have
\[[\bi_a,\bi_{b}]=0,\qquad \bi_{[a,b]}=[\bi_a,d_M\bi_b].\]
\end{definition}

%It is easy to see that if $\bi\colon L \to M$ is a Cartan homotopy, then the %map
%\[ \bl\colon L \to M, \qquad \bl_a=d_M\bi_a+\bi_{d_L a},\]
%is a morphism of DGLA and $\bi_{[a,b]}=[\bi_a,\bl_b]$

%\bigskip

\begin{lemma}\label{lem.identitadellecartan}
Let $\bi\colon L \to M$ be a Cartan homotopy and consider the degree 0
map
\[ \bl\colon L\to M,\qquad
\bl_a=d_M\bi_a+\bi_{d_L a}.
\]
Then:
\begin{enumerate}

\item $\bl$ is a morphism of DGLAs;

\item $\bi_{[a,b]}=[\bi_a,\bl_b]=(-1)^{\bar{a}}[\bl_a,\bi_b]$;

\item $[\bi_{[a,b]},\bl_c] -(-1)^{\bar{b}\;\bar{c}} [\bi_{[a,c]},\bl_b]+(-1)^{\bar{a}(\bar{b}+\bar{c})} [\bi_{[b,c]},\bl_a] =0$.

\end{enumerate}

\end{lemma}

\begin{proof} The proof of the first two items is straightforward.
For every $a,b$ and $c$, we have
\begin{multline*}
[\bi_{[a,b]},\bl_c] -(-1)^{\bar{b}\;\bar{c}} [\bi_{[a,c]},\bl_b]+(-1)^{\bar{a}(\bar{b}+\bar{c})} [\bi_{[b,c]},\bl_a]\\
=\bi_{[[a,b],c]} -(-1)^{\bar{b}\;\bar{c}} \bi_{[[a,c],b]}+(-1)^{\bar{a}(\bar{b}+\bar{c})} \bi_{[[b,c],a]}=0,
\end{multline*}
where the last  equality follows from
the graded Jacobi identity in $L$:
\begin{equation}\label{equ.Jacobicartanizzato}
[[a,b],c] -(-1)^{\bar{b}\;\bar{c}} [[a,c],b]+(-1)^{\bar{a}(\bar{b}+\bar{c})}[[b,c],a]=0.\end{equation}
Notice that the above form of Jacobi identity can be written as
\[\sum_{\sigma\in S(2,1)}\chi(\sigma)[[a_{\sigma(1)},a_{\sigma(2)}],a_{\sigma(3)}]=0,\]
where $\chi(\sigma)$ is the product of the signature and the Koszul sign
of the unshuffle $\sigma$.

\end{proof}

\begin{lemma}\label{lem mor infinito di DGLA}
Let $= (L,d_L, [ \, , \, ])$ and $= (M,d_M, [ \, , \, ])$ be  DGLAs. Consider the linear morphisms
\[
f_1: L \to M \mbox { and } f_2:  \bigwedge ^2 L \to M,
\]
where $f_1$ has degree zero and $f_2$ has degree $-1$  (see Remark~\ref{rem.morphismidecalati}). Then, the sequence $\{f_1, f_2,0,0, \cdots \}$ is an $L_\infty$ morphism of DGLAs if and only if the following equations are satisfied for every $a,b,c,d\in L$.
\begin{enumerate}
\item $f_1 \circ d_L = d_M \circ f_1$ (i.e., $f_1$ is a morphism of complexes);

\item $[f_1(a),f_1(b)]=f_1([a,b])-d_M \circ f_2(a,b)-f_2(d_L a,b)+(-1)^{\bar{a}\bar{b}} f_2(d_L b,a)$;

\item $f_2([a,b],c)- (-1)^{\bar{b}\bar{c}}f_2([a,c],b) +(-1)^{\bar{a}(\bar{b}+\bar{c})}f_2([b,c],a) =$\\
    $=(-1)^{a}[f_1(a),f_2(b,c)]-[f_2(a,b),f_1(c)]+
    (-1)^{\bar{b}\bar{c}}[f_2(a,c),f_1(b)]$;
\item $[f_2(a,b),f_2(c,d)] + (-1)^{(\bar{c} +1) ( \bar{b}+1)}[f_2(a,c),f_2(b,d)]=0$.
\end{enumerate}
 \end{lemma}

\begin{proof}
It follows from the explicit formula of $L_\infty$ morphism of DGLAs, given for instance in
\cite{KellerDefoQuant,LadaMarkl}.

\end{proof}

\begin{theorem}\label{thm.cartan.infinito}
Let $L,M$ be DGLAs and $\bi\colon L\to M$ be a Cartan homotopy. Then, the sequence of linear maps
\[ g_1\in \Hom^0(L,M[t,dt]),\qquad g_1(a)=t\bl_a+dt\bi_a,\]
\[ g_2\in \Hom^{-1}(\bigwedge^2 L, M[t,dt]),\qquad g_2(a,b)=t(1-t)\bi_{[a,b]},\]
\[ g_n=0\quad\forall n\ge 3,\]
defines an $L_{\infty}$ morphism $L\dashrightarrow M[t,dt]$ (see Remark~\ref{rem.morphismidecalati}).
\end{theorem}

\begin{proof}
We need to check the four conditions  of Lemma~\ref{lem mor infinito di DGLA}.
As regards (1), since $\bl$ is a morphism of DGLAs  and $d\bi+\bi d =\bl$ we have
\[
d_M(g_1(a))=dt\bl_a+ td(\bl_a) -dtd(\bi_a)= dt\bl_a+ t\bl_{da} -dt(\bl_a-\bi_{da})=t\bl_{da}+dt\bi_{da}=g_1(d_L(a)).
\]

As regards (2),  we have
\[g_1([a,b])-d_M \circ g_2(a,b)-g_2(d_L a,b)+(-1)^{\bar{a}\bar{b}} g_2(d_L b,a)=
\]
\[
t\bl_{[a,b]} +dt\bi_{[a,b]}
-d_M(t(1-t)\bi_{[a,b]} )-
 t(1-t)\bi_{[da,b]}
 + (-1)^{\bar{a}\bar{b}}t(1-t)
\bi_{[db,a]}=
\]
\[
t^2[\bl_a,\bl_b]+2tdt\bi_{[a,b]} =t^2[\bl_a,\bl_b]+tdt((-1)^{\bar{a}}[\bl_a,\bi_b]+[\bi_a, \bl_b])\]
\[
= [g_1(a),g_1(b)]\]
since we have
\[
-d_M(t(1-t)\bi_{[a,b]} )= -((1-2t)dt\bi_{[a,b]} + t(1-t)d(\bi_{[a,b]}))=
\]
\[
(2t-1)dt\bi_{[a,b]} + t(t-1)( \bl_{[a,b]}+\bi_{[da,b]}- (-1)^{\bar{a}\bar{b}}\bi_{[db,a]}).
\]

As regards Condition (3), we have that
 \[g_2([a,b],c)- (-1)^{\bar{b}\bar{c}}g_2([a,c],b) + (-1)^{\bar{a}(\bar{b}+\bar{c})}g_2([b,c],a) =
 \]
 \[
 (1-t)t(\bi_{[[a,b],c]}  -(-1)^{\bar{b}\bar{c}} \bi_{[[a,c],b]}
 + (-1)^{\bar{a}(\bar{b}+\bar{c})} \bi_{[[b,c],a]} )=0
 \]
 by Equation~\eqref{equ.Jacobicartanizzato} of Lemma \ref{lem.identitadellecartan}; moreover,
\[(-1)^{a}[g_1(a),g_2(b,c)]-[g_2(a,b),g_1(c)]+
 (-1)^{\bar{b}\bar{c}}[g_2(a,c),g_1(b)] \]
 \[
 =t^2(1-t ) ( (-1)^{a}[\bl_a,\bi_{[b,c]}] - [\bi_{[a,b]},\bl_c] + (-1)^{\bar{b}\bar{c}} [\bi_{[a,c]},\bl_b]  ) =0
    \]
by Lemma~\ref{lem.identitadellecartan}.
Finally, condition (4) follows from the fact that $[\bi_x,\bi_y]=0$, for every choice of $x, y \in L$.
\end{proof}

\begin{corollary}\label{cor.cartan.infinito}
In the same assumption of Theorem~\ref{thm.cartan.infinito}, let $N\subset M$ be a differential graded Lie subalgebra such that $\bl(L)\subset N$ and 
\[ TW(\chi)=\{(x,y(t))\in N\times M[t,dt]\mid y(0)=0,\; y(1)=x\}\]
the homotopy fiber of the inclusion $\chi \colon N\hookrightarrow M$.
Then, the maps
\[ g_1\in \Hom^0(L,TW(\chi)),\qquad g_1(a)=(\bl_a,t\bl_a+dt\bi_a),\]
\[ g_2\in \Hom^{-1}({\bigwedge}^2 L, TW(\chi)),\qquad g_2(a,b)=(0,(1-t)t\bi_{[a,b]}),\]
\[ g_n=0\quad\forall n\ge 3,\]
define an $L_{\infty}$ morphism $L\dashrightarrow TW(\chi)$.
\end{corollary}

\begin{proof} Since $ TW(\chi) \subset N\times M[t,dt]$, we need to
check the 4 conditions  of Lemma~\ref{lem mor infinito di DGLA} on both components.
As regard the first component, Conditions (1) and (2) follows  from the fact that $\bl$ is a morphism of DGLAs (and so it commutes with differentials and brackets).
Conditions (3) and (4) are trivial, since ${g_2}_{|N}=0$.
As regard the second components, it is due to the previous theorem.

\end{proof}

The following definition is the natural modification of the definition
of Tamarkin-Tsygan calculus \cite{TT05}, whenever the Gerstenhaber algebra is replaced by a DGLA.

\begin{definition}\label{def.contraction}
Let $L$ be a differential graded Lie algebra and $V$ a differential
 graded vector space. A bilinear map
\[ L\times V\xrightarrow{\quad\contr\quad} V\]
of degree $-1$ is called a \emph{calculus} if the induced map
\[ \bi\colon L\to \Hom^*_{\K}(V,V),\qquad \bi_l(v)=l\contr v,\]
is a Cartan homotopy.\end{definition}

The notion of calculus is stable under scalar extensions, more
precisely, we have the following result.

\begin{lemma}\label{lem. cartan implica cartan XA}

Let $V$ be  a differential graded vector space  and
\[
L\times V\xrightarrow{\quad\contr\quad} V
\]
a calculus. Then, for every differential graded commutative
algebra $A$, the natural extension
\[
(L\otimes A)\times (V\otimes A) \xrightarrow{\quad\contr\quad}
(V\otimes A) \qquad (l\otimes a)\contr (v\otimes
b)=(-1)^{\bar{a}\;\bar{v}}l \contr v\otimes ab,
\]
is a calculus.
\end{lemma}

\begin{proof}
Straightforward, see \cite[Lemma 4.7]{algebraicBTT}.
\end{proof}

The notions of Cartan homotopy and calculus extend naturally to
the semicosimplicial objects and sheaves. Here, we consider only the case of
calculus.

\begin{definition}\label{def.contractioncosimpl}
Let $\mathfrak{g}^\Delta$ be a semicosimplicial DGLA and
$V^\Delta$ a semicosimplicial differential graded vector space.
A semicosimplicial Lie-calculus
\[ \mathfrak{g}^\Delta\times V^\Delta\xrightarrow{\;\contr\;} V^\Delta,\]
is a sequence of calculi $\mathfrak{g}_n\times V_n\xrightarrow{\;
 \contr\;} V_n$, $n\ge 0$,
commuting with coface maps, i.e., $\de_k(l\contr v)=\de_k(l)\contr
\de_k(v)$, for every $k$.
\end{definition}

\begin{lemma}\label{lem.TWforcontractions}
Every semicosimplicial calculus
\[ \mathfrak{g}^\Delta\times V^\Delta\xrightarrow{\;\contr\;} V^\Delta\]
extends  naturally to a calculus
\[  {TW}(\mathfrak{g}^\Delta)\times  {TW}(V^\Delta)
\xrightarrow{\;\contr\;}  {TW}(V^\Delta).\]
\end{lemma}

\begin{proof}
Straightforward, see
 \cite[Proposition 4.9]{algebraicBTT}.
\end{proof}

\bigskip
\section{Calculus on de Rham complex of a DG-scheme}

\medskip

Let $S$ be a DG-algebra and  $d\colon S\to \Omega_{S/\K}$ its  universal derivation.
Denote by $\Omega^0_S=S$, $\Omega^1_S=\Omega_{S/\K}[-1]$ and
\[\Omega^k_S=\Sym_S^k \Omega^1_S.\]
The $S$-module $\Omega^1_S$ is generated by the elements $da$, $a\in S$ and
$\deg(da)=\deg(a)+1$.
In  the algebra $\Omega^*_S=\bigoplus_k \Omega^k_S$ we have
\[ da\wedge db=(-1)^{(\bar{a}+1)(\bar{b}+1)}db\wedge da.\]

The $\K$-linear map of degree $+1$ $d\colon S\to \Omega^1_S$ extends to a unique differential
(de Rham) of the graded commutative algebra $\Omega^*_S$, i.e.,
$d\colon \Omega^*_S\to \Omega^*_S$.
Notice that $d\in \Der_{\K}^1(\Omega_S^*,\Omega_S^*)$, $d^2=0$ and
$d\Omega^i_S\subset\Omega^{i+1}_S$.

Next, assume that $\alpha\in \Der_{\K}^k(S,S)$ is a
$\K$-linear derivation. Then  $\alpha$ induces a morphism
\[ \bi_\alpha\in\Hom_S^{k-1}(\Omega^1_S,S)=\Hom_S^{k}(\Omega_{S/\K},S)\]
such that $\bi_{\alpha}(da)=\alpha(a)$, for every $a\in S$.
By Leibniz rule, $\bi_{\alpha}$ extends to a $S$-derivation of the graded symmetric algebra
$\Omega^*_S$, i.e.,
\[\bi_\alpha\in\Der_S^{k-1}(\Omega^*_S,\Omega^*_S).\]
We shall call $\bi_{\alpha}\colon \Omega^*_S\to \Omega^*_S$ the \emph{interior product}
by $\alpha$.
Notice that $\bi_\alpha(\Omega^k_S)\subset \Omega^{k-1}_S$.
Define the Lie derivation
\[ L_{\alpha}:=[\bi_{\alpha},d]\in \Der^k_{\K}(\Omega^*_S,\Omega^*_S).\]
Since $L_{\alpha}$ is $[-,d]$-exact it is also $[-,d]$-closed, i.e., $[L_{\alpha},d]=0$; if
$a\in S=\Omega^0_S$ we have
\[ L_{\alpha}(a)=\bi_{\alpha}(da)=\alpha(a),\qquad
L_{\alpha}(da)=(-1)^{k}d(L_\alpha(a))=(-1)^{k}d(\alpha(a)).\]
Equivalently,  $L_{\alpha}$ is the unique derivation extending
$\alpha\colon \Omega^0_S\to \Omega^0_S$,
such that $\deg(L_{\alpha})=\deg(\alpha)$ and $[d,L_\alpha]=0$. Notice that $L_{\alpha}(\Omega^i_S)
\subset \Omega^i_S$.

If $\delta\in \Der_{\K}^1(S,S)$ is a
differential, then  also
$L_\delta\colon \Omega^*_S\to \Omega^*_S$ is a differential: in fact, for every $a\in S$
\[ L_{\delta}^2(a)=\delta^2(a)=0,\qquad L_{\delta}^2(da)=L_\delta(-d\delta(a))=
d\delta^2(a)=0.\]
In such case,
$(\Omega^*_S,d,L_\delta)$ is a double complex and $(d+L_\delta)^2=0$.

We conclude the section, describing some properties of $\bi_{\alpha}$ and $L_\alpha$,
for $\alpha\in \Der_{\K}^k(S,S)$.

\begin{lemma}\label{lem. Cartan equalities}
Given $\alpha,\beta\in\Der_{\K}^*(S,S)$,  we have the equalities:\begin{enumerate}
\item $[\bi_{\alpha},\bi_\beta]=0$;
\item $L_{\beta}=[\bi_{\beta}, d]=(-1)^{\bar{\beta}}[d,\bi_{\beta}]$;
\item $\bi_{[\alpha,\beta]}=[L_{\alpha},\bi_{\beta}]=[[\bi_{\alpha},d],\bi_{\beta}]$;

\item $[L_{\alpha},d]=0$, $L_{[\alpha,\beta]}=[L_{\alpha},L_{\beta}]$;

\item $[L_{\delta},\bi_{\beta}]+\bi_{[\delta,\beta]}=0$, for every $\delta\in \Der^1_{\K}(S,S)$;

\item $\bi_{[\alpha,\beta]}=[[\bi_{\alpha},d+L_{\delta}],\bi_{\beta}]=
[\bi_{\alpha},[[d+L_{\delta}],\bi_{\beta}]]$, for every $\delta\in \Der^1_{\K}(S,S)$.

\end{enumerate}
\end{lemma}

\begin{proof}
Since $\bi_{\alpha}\bi_{\beta}(da)=0$, for every $\alpha,\beta$, we have
$[\bi_{\alpha},\bi_\beta]=0$.

Next, as regard (3), we have
\[ \bi_{[\alpha,\beta]}(da)=[\alpha,\beta]a=\alpha(\beta(a))-(-1)^{\bar{\alpha}\,
\bar{\beta}}\beta(\alpha(a))=[L_\alpha,L_\beta]a\]
and
\[ [L_{\alpha},\bi_{\beta}](da)=L_{\alpha}\bi_{\beta}(da)-(-1)^{(\bar{\beta}-1)
\bar{\alpha}}\bi_{\beta}L_{\alpha}(da)=L_{\alpha}L_{\beta}(a)-(-1)^{\bar{\beta}\;
\bar{\alpha}}\bi_{\beta}d(L_{\alpha}a)=[L_\alpha,L_\beta]a.\]

As regard (4), we have
\[ L_{[\alpha,\beta]}=[\bi_{[\alpha,\beta]},d]=[[L_{\alpha},\bi_{\beta}],d]=
[L_{\alpha},[\bi_{\beta},d]]=[L_{\alpha},L_{\beta}].
\]

As regard (5), for every  $\delta\in \Der^1_{\K}(S,S)$, we have
\[\bi_{[\delta,\beta]}=-(-1)^{\bar{\beta}}\bi_{[\beta,\delta]}=
(-1)^{\bar{\beta}}[\bi_{\beta},L_{\delta}]=-[L_{\delta},\bi_{\beta}].
\]
The last one follows from
\[ [[\bi_{\alpha},L_{\delta}],\bi_{\beta}]=\pm[\bi_{[\alpha,\delta]},\bi_{\beta}]=0.\]

\end{proof}

\begin{proposition} Let $S$ be a DG-algebra and $\delta\in \Der^1_{\K}(S,S)$  a differential. Then,
 the map
\[ \bi\colon (\Der^*_{\K}(S,S), [\delta,-])\to
(\Der^*_{\K}(\Omega^*_S,\Omega^*_S), [d+L_{\delta},-]),\]
is a Cartan homotopy with
$[d+L_{\delta},\bi_{\alpha}]+\bi_{[\delta,\alpha]}=(-1)^{\bar{\alpha}} L_{\alpha}$.
\end{proposition}

\begin{proof}
By Lemma \ref{lem. Cartan equalities}, we have
\[[d+L_{\delta},\bi_{\alpha}]+\bi_{[\delta,\alpha]}=[d,\bi_{\alpha}]=
(-1)^{\bar{\alpha}}L_{\alpha}\]
and then
\[\bi_{[\alpha,\beta]}=[\bi_{\alpha},[d,\bi_{\beta}]]=
[\bi_{\alpha},[d+L_{\delta},\bi_{\beta}]+\bi_{[\delta,\beta]}] =
[\bi_{\alpha},[d+L_{\delta},\bi_{\beta}]]
.\]

\end{proof}

Assume that $S=\bigoplus_{i\le 0} S^i$ and let $T\subset S^0$ be a  multiplicative subset.
Then, for  any $S$-DG-module $M$, every $\K$-derivation $S\to M$ extends in a unique way to a derivation
$T^{-1}S\to M$; therefore, $\Omega_{T^{-1}S/\K}=T^{-1}\Omega_{S/\K}$,
\[ \Omega^*_{T^{-1}S}=T^{-1}\Omega^*_S.\]
It is easy to verify  that the two natural maps
\[ \Der_{\K}^*(S,S)\to \Der_{\K}^*(T^{-1}S,T^{-1}S),\qquad
\Der_{\K}^*(\Omega^*_S,\Omega^*_S)\to \Der_{\K}^*(\Omega^*_{T^{-1}S},\Omega^*_{T^{-1}S}),\]
commute with the de Rham differential, interior products, Lie derivative and Cartan formulas, giving therefore a canonical morphism  of calculi
\[\xymatrix{\Der_{\K}^*(S,S)\times \Omega^*_S\ar[r]\ar[d]&\Omega^*_S\ar[d]\\
\Der_{\K}^*(T^{-1}S,T^{-1}S)\times \Omega^*_{T^{-1}S}\ar[r]&\Omega^*_{T^{-1}S}}\]
This implies that, for every DG-scheme $(E,\sS_E)$ over the field $\K$, the interior
 product induces a morphism of sheaves
\[ Der^*_{\K}(\sS_E,\sS_E)\times \Omega^*_{\sS_E}\to \Omega^*_{\sS_E}\]
which is still a calculus, where $\Omega^*_{\sS_E}$ is the  exterior algebra of the
quasi-coherent sheaf of differentials on $E$.

\bigskip

\bigskip
\section{Local cohomology of DG-sheaves}

The theory of local cohomology extends naturally to DG-sheaves over DG-schemes. In order to fix the notation and prove some homotopy invariant properties, we briefly recall   its construction in the quasi-coherent case, using the approach of \v{C}ech complex \cite{BS}.

Let $A=\bigoplus_{i}A^i$ be a fixed DG-ring.
Given an element $a\in Z^0(A)$ and an $A$-module $M=\oplus_i M^i$, we will denote
$M_a=\oplus_i M^i_a$ the localization of $M$ by the multiplicative subset of powers of $a$.

Since the localization is an exact functor on the category of $Z^0(A)$-modules,
we have $H^i(M_a)=H^i(M)_a$, for every $i\in \Z$.

Let $\underline{a}=(a_1,\ldots,a_n)$ be a sequence of elements of $Z^0(A)$; for every $A$-module $M$ and every subset $H=\{h_1<h_2<\cdots<h_j\}
\subseteq \{1,\ldots,n\}$,  denote by
\[ a_{\emptyset}=1,\quad a_H=a_{h_1}\cdots a_{h_j}.\]
Notice that $a_{H\cup K}|a_H a_K|a_{H\cup K}^2$, for every $H,K\subseteq \{1,\ldots,n\}$; therefore, $(M_{a_H})_{a_K}=M_{a_{H\cup K}}$ and if $H\subset K$ there exists a natural localization map
$M_{a_H}\to M_{a_K}$.

\begin{definition} In the above situation, denote by $\check{C}^*(a_1,\ldots,a_n;M)$ the  complex of $A$-modules

\begin{equation}\label{equ.cech}
0\to \check{C}^0(a_1,\ldots,a_n;M)\mapor{\delta}\check{C}^1(a_1,\ldots,a_n;M)\mapor{\delta}\cdots\to \check{C}^n(a_1,\ldots,a_n;M)\to 0,
\end{equation}
where
\[ \check{C}^0(a_1,\ldots,a_n;M)=M,\qquad
\check{C}^i(a_1,\ldots,a_n;M)=\prod_{H\subset \{1,\ldots,n\}, |H|=i} M_{a_H}\]
and
\[ \delta(m)_{h_0,h_1,\ldots,h_r}=\sum_{i}(-1)^im_{h_0,h_1,\ldots,\widehat{h_i},\ldots,h_r}.\]
is the \v{C}ech differential.
\end{definition}

\begin{definition} We will denote by
\[H^i_{\underline{a}}(M)=H^i(a_1,\ldots,a_n;M)\]
the $i$-th cohomology  group of the complex
$\check{C}^*(a_1,\ldots,a_n;M)$: notice that $H^i(a_1,\ldots,a_n;M)$ is an $A$-module.
\end{definition}

\begin{remark}\label{Remark.local_cohomolo-comology_supported_ideal}
Let $A=\bigoplus_{i}A^i$ be a fixed DG-ring and consider the scheme $Spec(Z^0(A))$.
For every $Z^0(A)$-module $F$, we can consider the associated quasi-coherent
sheaf $\widetilde{F}$ on $Spec(Z^0(A))$. Then, 
the complex $\check{C}^*(a_1,\ldots,a_n;F)$ computes the local cohomology of
$\widetilde{F}$  with  support on the ideal generated by $a_1,\ldots,a_n$.
\end{remark}

Every
permutation $\sigma$ of $\{1,\ldots,n\}$ induces
a natural isomorphism (see e.g. \cite{BS})
\[ \check{\sigma}\colon \check{C}^*(a_1,\ldots,a_n;M)\to
\check{C}^*(a_{\sigma(1)},\ldots,a_{\sigma(n)};M).\]
For later use, we point out that, via the natural isomorphism
$M_{a_1\cdots a_n}=M_{a_{\sigma(1)}\cdots a_{\sigma(n)}}$, the map
\[ \check{\sigma}\colon \check{C}^n(a_1,\ldots,a_n;M)\to
\check{C}^n(a_{\sigma(1)},\ldots,a_{\sigma(n)};M)\]
is simply given by multiplication by the signature of $\sigma$.

Moreover,
\[\check{C}^i(a_1,\ldots,a_n;M)=\check{C}^i(a_1,\ldots,a_n;A)\otimes_A M=
\check{C}^i(a_1,\ldots,a_n;Z^0(A))\otimes_{Z^0(A)} M\;.\]

Given any $b\in Z^0(A)$ and any integer $i$, we have a natural projection morphism of $A$-modules
\[\check{C}^i(a_1,\ldots,a_n,b; M)\to
\check{C}^i(a_1,\ldots,a_n; M)\to 0\;.\]
The same argument used in \cite[p. 86]{BS} shows that, if $b\in
\sqrt{(a_1,\ldots,a_n)}\subset Z^0(A)$, then the projection map
$\check{C}^\ast(a_1,\ldots,a_n,b; M)\to
\check{C}^\ast(a_1,\ldots,a_n; M)$ is a quasi-isomorphism in the category of complexes of $A$-modules.

\begin{theorem} The isomorphism class of the cohomology modules
\[ H^i(a_1,\ldots,a_n;M)\]
depends only on the ideal $\sqrt{(a_1,\ldots,a_n)}\subset Z^0(A)$.
More precisely, if
\[\sqrt{(a_1,\ldots,a_n)}=\sqrt{(b_1,\ldots,b_m)}\subset Z^0(A),\]
then there exists a \emph{canonical} isomorphism
\[ H^i(a_1,\ldots,a_n;M)\cong H^i(b_1,\ldots,b_m;M). \]
\end{theorem}

\begin{proof}
We have
\[
\xymatrix{
\check{C}^*(a_1,\ldots,a_n,b_1,\ldots,b_m;M)\ar[d]\ar[rr]^{\;\;\text{permutation}\;\;}_{\simeq}&&\check{C}^*(b_1,\ldots,b_m,a_1,\ldots,a_n;M)\ar[d]\\
\check{C}^*(a_1,\ldots,a_n;M)&&
 \check{C}^*(b_1,\ldots,b_m;M)}
\]
where the vertical maps are  the
surjective quasi-isomorphisms induced by projections.
\end{proof}

The complex of $A$-modules $\check{C}^*(a_1,\ldots,a_n;M)$ can be clearly interpreted as a double complex of $Z^0(A)$-modules.
We will denote by
$\H^*_{\underline{a}}(M)=\H^*(a_1,\ldots,a_n;M)$
the cohomology  of the associated total complex
$\tot(\check{C}^*(a_1,\ldots,a_n;M))$.

Thus, we have
\[ \tot(\check{C}^*(a_1,\ldots,a_n;M))=\tot(\check{C}^*(a_1,\ldots,a_n;A))\otimes_A M=\tot(\check{C}^*(a_1,\ldots,a_n;Z^0(A)))\otimes_{Z^0(A)} M
\]
and, since $\check{C}^*(a_1,\ldots,a_n;M)$ is a complex of finite length, we have two convergent
 spectral sequences: the former is
\[ E^{i,j}_2=H^i(H^j(a_1,\ldots,a_n;M))\Rightarrow \H^{i+j}(a_1,\ldots,a_n;M),\]
and, since the localization is an exact functor, the latter is
\[ E^{i,j}_2=H^i(a_1,\ldots,a_n;H^j(M))\Rightarrow \H^{i+j}(a_1,\ldots,a_n;M).\]

\begin{lemma}\label{lem.localcohomologylci}
Assume that $Z^0(A)$ is a Noetherian ring;
if $a_1,\ldots,a_n$ is a regular sequence in $Z^0(A)$ and $M$ is a complex of free $Z^0(A)$-modules, then
\[ \H^i(a_1,\ldots,a_n;M)=H^{i-n}(H^n(a_1,\ldots,a_n;M)),\]
for every $i\in \Z$.
\end{lemma}

\begin{proof}
As observed in Remark \ref{Remark.local_cohomolo-comology_supported_ideal},
for every $Z^0(A)$-module $F$, the complex
$\check{C}^*(a_1,\ldots,a_n;F)$ computes the local cohomology with
 support on the ideal generated by $a_1,\ldots,a_n$ of the quasi-coherent
sheaf $\widetilde{F}$ on $Spec(Z^0(A))$; therefore,
$H^i(a_1,\ldots,a_n;Z^0(A))=0$ for every $i\not=n$ \cite[Thm. 3.8]{GrothendieckLC}.
If $F$ is free then also $H^i(a_1,\ldots,a_n;F)=0$ for every $i\not=n$,
since it is a direct sum of copies of $H^i(a_1,\ldots,a_n;Z^0(A))$.
\end{proof}

Next, let us consider the functorial properties of local cohomology.

First, if  $f\colon M\to N$ a morphism of $A$-modules,
then $f$ induces a morphism of complexes of  $A$-modules
\[ f_*\colon \check{C}^*(a_1,\ldots,a_n;M)\to \check{C}^*(a_1,\ldots,a_n;N)\]
and, therefore,  morphisms
\[ f_i\colon H^i(a_1,\ldots,a_n;M)\to H^i(a_1,\ldots,a_n;N),\qquad
f_i\colon \H^i(a_1,\ldots,a_n;M)\to \H^i(a_1,\ldots,a_n;N).\]

\begin{lemma}\label{lem.homotopyinvariance1}
In the notation above, if $f\colon M \to N$ is a quasi-isomorphism,
then
\[f_i\colon \H^i(a_1,\ldots,a_n;M)\to \H^i(a_1,\ldots,a_n;N)\]
is an isomorphism for every $i$.
\end{lemma}

\begin{proof} Immediate consequence of the spectral sequence
\[ H^j(a_1,\ldots,a_n;H^i(M))\Rightarrow \H^{i+j}(a_1,\ldots,a_n;M).\]
\end{proof}

Let $g\colon A\to B$ be a morphism of DG-rings, then for every $A$-Module
$M$, setting $N=M\otimes _A B$, we have  natural maps $M\to N$,  $M_{a_H}\to M\otimes _A B_{g(a_H)}$,
and, therefore, a morphism of complexes
\[ g_*\colon \check{C}^*(a_1,\ldots,a_n;M)\to \check{C}^*(g(a_1),\ldots,g(a_n);M\otimes _A B)\]
inducing morphisms of $A$-modules
\[ g_*\colon H^i(a_1,\ldots,a_n;M)\to H^i(g(a_1),\ldots,g(a_n);M\otimes _A B).\]

\begin{lemma}\label{lem.homotopyinvariance2}
If $g\colon A\to B$ is a quasi-isomorphism of  DG-rings and
$M$ is a semifree  $A$-module,
then
\[ g_*\colon \H^i(a_1,\ldots,a_n;M)\to \H^i(g(a_1),\ldots,g(a_n);M\otimes_A B).\]
is an isomorphism for every $i$. \end{lemma}

\begin{proof} For every
$a\in Z^0(A)$, we have  a quasi-isomorphism
$g\colon A_a\to B_{g(a)}$; then, since $M$ is semifree, we also have a quasi-isomorphism
\[ M_a=A_a\otimes_A M\to B_{g(a)}\otimes_A M.\]
By Lemma \ref{lem.homotopyinvariance1}, this induce quasi-isomorphisms
\[ \check{C}^i(a_1,\ldots,a_n;M)\to \check{C}^i(g(a_1),\ldots,g(a_n);M\otimes _A B),\]
\[   \tot(C^*(a_1,\ldots,a_n;M))\to \tot(C^*(g(a_1),\ldots,g(a_n);M\otimes_A B)).\]
\end{proof}

\begin{remark}\label{rem.calculus.localcohomology}
The calculus
$\Der^*_{\K}(A,A)\times \Omega^*_A\to \Omega^*_A$ induced by the interior product
and the de Rham differential
commutes with local cohomology construction. In fact, since it commutes with localizations,
 we have
for every
$a_1,\ldots,a_n\in Z^0(A)$ a sequence of calculi
\[ \Der^*_{\K}(A,A)\times C^j(a_1,\ldots,a_n,\Omega^*_A)\to C^j(a_1,\ldots,a_n,\Omega^*_A),
\qquad j=0,\ldots,n,\]
commuting with \v{C}ech differentials and so inducing a sequence of calculi
\[ \Der^*_{\K}(A,A)\times H^j(a_1,\ldots,a_n;\Omega^*_A)\to H^j(a_1,\ldots,a_n;\Omega^*_A),
\qquad j=0,\ldots,n.\]
\end{remark}

\begin{remark}
Since the construction of $H^i(a_1,\ldots,a_n;M)$ commutes with localization, we have
\[ H^i(a_1,\ldots,a_n;M_b)=H^i(a_1,\ldots,a_n;M)_b,\qquad \forall \  b\in Z^0(A),\]
and therefore, if $\sF$ is any quasi-coherent sheaf on a DG-scheme $(T,\sA)$, and $Y\subset T$
is any closed subscheme of $T$, we have the local cohomology sheaves
$\sH_Y^i(T,\sF)$ defined locally as $H^i(a_1,\ldots,a_n,\sF)$, where $a_1,\ldots,a_n$ is a set of generators of the ideal sheaf of $Y$ in $T$.
This applies in particular for the quasi-coherent DG-sheaves $\Omega^k_{\sA}$.
\end{remark}

\bigskip
\section{The derived cycle class of $Z$}

Let us go back to the  situation described in the introduction, i.e., $Z\subset X$ local complete intersection of codimension $p$, defined as the zero locus of a section $f$ of a vector bundle $\pi \colon E\to U$ of rank $p$.
As in \cite{bloch}, we have the local cohomology sheaves $\sH^i_Z(X,\Omega^j_X)$ and
a canonical cycle class
\[ \{Z\}\in
\H^{2p}(X,\sH^p_Z(X,\Omega^p_X)\mapor{d}\sH^p_Z(X,\Omega^{p+1}_X)\to\cdots),
\]
where $\sH^p_Z(X,\Omega^j_X)$ is a DG-sheaf concentrated in degree $p+j$. Recall that, since $Z$ is a locally complete intersection and $\Omega^j_X$ is locally free, then $\sH^i_Z(X,\Omega^j_X)=0$ for every $i\not=p$.  The image of $\{Z\}$ under the injective map
\[ \H^{2p}(X,\sH^p_Z(X,\Omega^p_X)\mapor{d}\sH^p_Z(X,\Omega^{p+1}_X)\to\cdots)
\hookrightarrow \H^{2p}(X,\sH^p_Z(\Omega^p_X))=\Gamma(X,\sH^p_Z(\Omega^p_X))
\]
is the global section
$\{Z\}'\in \Gamma(X,\sH^p_Z(\Omega^p_X))=\Gamma(U,\sH^p_Z(\Omega^p_U))$ given in local coordinates by
\[ \{Z\}'=\frac{df_1\wedge\cdots\wedge df_p}{f_1\cdots f_p}\in
\check{C}^p(f_1,\ldots,f_p;\Omega_X^p),\]
where $f_1,\ldots,f_p$ are the components of $f$ with respect to a local trivialization of the bundle $E$.
By local duality,  the section $\{Z\}'$ is non trivial at every point of $Z$.

Similarly, working  in the DG-scheme $(E,\sS)$, if $U^0\subset E$ is the image of the zero section $U\to E$, we have a canonical cycle class $\{U^0\}'\in \Gamma(E,\sH^p_{U^0}(\Omega^p_{\sS}))$,
given in local coordinates by
\[ \{U^0\}'=\frac{dz_1\wedge\cdots\wedge dz_p}{z_1\cdots z_p},\]
where $z_1,\ldots,z_p$ are a set of linear coordinates in the fibers of the bundle $E$. The same proof of the classical case shows that $\{U^0\}'$ is independent from the choice of the coordinate system.

Since $U^0$ is smooth of codimension $p$ in $E$ and the DG-sheaves $\Omega^j_{\sS}$ are semifree, we have $\sH^i_{U^0}(E,\Omega^j_{\sS})=0$ for every $j$ and every $i\not=p$.

\begin{definition} In the above notation,
for every $j\ge 0$, we denote $\sK^j=\pi_* \sH^p_{U^0}(E,\Omega^j_{\sS})$.\end{definition}

Every $\sK^j$ is a quasi-coherent DG-sheaf on $U$.

\begin{lemma}\label{lem. quasi iso K cohom locale} For every $j\ge 0$, the section $f$ induces  a  quasi-isomorphism of DG-sheaves on $U$
\[ \sK^j\to\sH^p_{Z}(U,\Omega^j_U)=\sH^p_{Z}(X,\Omega^j_X),\]
with $\sH^p_{Z}(U,\Omega^j_U)$   concentrated in degree $p+j$.
\end{lemma}

\begin{proof}
The question is local, therefore we may assume $U=Spec(P)$ affine and $E$ trivial on $U$, i.e., $E=Spec(P[z_1,\ldots,z_p])$.
Thus
\[ \Gamma(E,\sS)=S=P[z_1,\ldots,z_p,y_1,\ldots,y_p]\qquad\text{with}\quad \deg(y_i)=-1\quad d(y_i)=f_i-z_i.\]
The map, induced by $f$:
\[ S\to P,\qquad y_i\mapsto 0,\quad z_i\mapsto f_i\]
is a quasi-isomorphism and then, by Lemma~\ref{lem.homotopyinvariance2} we have
isomorphisms of $P$-modules
\[\H_{U^0}^i(E,\Omega^j_S)\to \H_{Z}^i(U,\Omega^j_S\otimes_S P),\]
that, according to Lemma~\ref{lem.localcohomologylci} gives a quasi-isomorphism of DG-sheaves
\[\sK^j=\pi_*\sH_{U^0}^p(E,\Omega^j_{\sS})\to \sH_{Z}^p(U,\Omega^j_{\sS}\otimes_{\sS} \Oh_U).\]
On the other hand, since $U$ is smooth, by general properties of the cotangent complex,
the natural map $\Omega^1_S\otimes_S P\to \Omega^1_P$
is a quasi-isomorphism, and also  a homotopy equivalence since $\Omega^1_S\otimes_S P$ and
$\Omega^1_P$ are semifree. This implies that every morphism $\Omega^j_S\otimes_S P\to \Omega^j_P$ is a homotopy equivalence; then, by Lemma~\ref{lem.homotopyinvariance1}  we get   isomorphisms
\[\H_{Z}^i(U,\Omega^j_S\otimes_S P)\to \H_{Z}^i(U,\Omega^j_P)\]
and so a quasi-isomorphism of DG-sheaves
\[ \sH_{Z}^p(U,\Omega^j_{\sS}\otimes_{\sS} \Oh_U)\to \sH_{Z}^p(U,\Omega^j_U).\]
\end{proof}

It is immediate to observe that the map
$\Gamma(U,\sK^p)\to\Gamma(U,\sH^p_{Z}(\Omega^p_U))$
sends the cycle class $\{U^0\}'$ into the cycle class $\{Z\}'$. In particular, $\{U^0\}'$ gives a non trivial cohomology class in $\H^{2p}(U,\sK^p)$ and, since $\{U^0\}'$ is annihilated by the de Rham differential, it gives also a non trivial hypercohomology class
\[ [\{U^0\}']\in \H^{2p}(U,\sK^{\ge p})=\H^{2p}(U,\sK^p\mapor{d}\sK^{p+1}
\mapor{d}\cdots).\]

\bigskip
\section{Proof of the main theorem}

Finally, we are ready for the proof of the main theorem. Recall that $\sK^j=\pi_* \sH^p_{U^0}(E,\Omega^j_{\sS})$, $\sM=\DER_{\Oh_U}^*(\Sym^*_{\Oh_U}(\sE^\vee[1]\oplus \sE^\vee),\Sym^*_{\Oh_U}(\sE^\vee[1]\oplus \sE^\vee))$  and
$\sM_{\perp}\subset \sM$ is  the subsheaf of derivations preserving the ideal generated by $\sE^\vee$.

The de Rham differential induces a map $d\colon \sK^*\to \sK^*$ of degree $+1$
where $\sK^*=\oplus_i \sK^i=\oplus_i  \pi_* \sH^p_{U^0}(E,\Omega^i_{\sS})$, inducing a structure of double differential graded sheaf on
$\sK^*$.

Then, according to Remark~\ref{rem.calculus.localcohomology},
the internal product induces a calculus
\[ \sM\times \sK^*\to \sK^*.\]

Passing to an affine open cover $\sU$, by Lemma \ref{lem.TWforcontractions}, we get a calculus
\[ TW(\sU,\sM)\times TW(\sU,\sK^*)\to TW(\sU,\sK^*),\quad \text{where}\quad TW(\sU,\sK^*)=
\oplus_j TW(\sU,\sK^j).\]

For simplicity of notation, we still denote by $\gamma\in
TW(\sU,\sK^{p})$
the image of the canonical cycle class $\{U^0\}'$ under the natural map
$\Gamma(U,\sK^{p})\to TW(\sU,\sK^{p})$ described in
Lemma~\ref{lem.naturalmap}; we will also denote by $TW(\sU,\sK^{\ge p})=\oplus_{j\ge p} TW(\sU,\sK^j)$.

Next, consider the following four differential graded Lie algebras
\[ L=\Hom^*_{\K}(TW(\sU,\sK^*),TW(\sU,\sK^*)),\qquad
L(\gamma_{\perp})=\{f\in L\mid f(\gamma)=0\},
\]
\[ \bar{L}=\{f\in L\mid f(TW(\sU,\sK^{\ge p}))\subset TW(\sU,\sK^{\ge p})\},\qquad
\bar{L}(\gamma_{\perp})=L(\gamma_{\perp})\cap \bar{L},\]
they give a commutative diagram of inclusions
\[\xymatrix{\bar{L}(\gamma_{\perp})\ar[d]^{j_\perp}\ar[r]^{\bar{\rho}}&\bar{L}\ar[d]^j\\
L(\gamma_{\perp})\ar[r]^{\rho}&L.}\]
Note that this diagram induces a morphism of DGLAs $TW(j_{\perp}) \to TW(j)$.

The morphism $\bl$ induced by the Cartan homotopy
\[ \bi\colon  TW(\sU,\sM)\to L=\Hom^*_{\K}(TW(\sU,\sK^*),TW(\sU,\sK^*))\]
preserves the subcomplex $TW(\sU,\sK^{\ge p})=\oplus_{j\ge p} TW(\sU,\sK^j)$, i.e.,
$\bl(TW(\sU,\sM))\subset \bar{L}$, and then
 Corollary~\ref{cor.cartan.infinito}
gives a canonical $L_{\infty}$ morphism
\[ \alpha_{\infty}\colon TW(\sU,\sM)\dashrightarrow TW(j).\]
By Remark~\ref{rem.quasiisoTWcono} the cohomology of $TW(j)$ is the same as the cohomology of
\[ \Hom^*_{\K}(TW(\sU,\sK^{\ge p}),TW(\sU,\sK^{<p}))[-1],\quad \text{where}\quad TW(\sU,\sK^{<p})=
\frac{TW(\sU,\sK^{*})}{TW(\sU,\sK^{\ge p})}.\]
Analogously, since
\[
\bi (TW(\sU,\sM_{\perp}))\subset L(\gamma_\perp),\quad \text{and}\quad
\bl(TW(\sU,\sM_{\perp}))\subset \bar{L}(\gamma_\perp),\]
we have a canonical $L_{\infty}$ morphism
\[ \alpha_{\infty}\colon TW(\sU,\sM_\perp)\dashrightarrow TW(j_\perp).\]
Therefore, we have a commutative diagram of $L_{\infty}$ morphisms
\begin{equation}\label{equ.quadrato}
\xymatrix{TW(\sU,\sM_{\perp})\ar[d]^{\chi} \ar@{-->}[r]^{\alpha_{\infty}}  &TW(j_\perp)\ar[d]^{\rho}\\
 TW(\sU,\sM)\ar@{-->}[r]^{\alpha_{\infty}} &TW(j)}\end{equation}

The cohomology of $TW(j_\perp)$ is the same as the cohomology of
\[ \{f\in \Hom^*_{\K}(\sU,TW(\sK^{\ge p}),TW(\sU,\sK^{<p}))[-1]\mid f(\gamma)=0\}.\]
Since $\gamma$ is non trivial in the cohomology of $TW(\sU,\sK^{\ge p})$, we have that the inclusion
$\rho\colon TW(j_\perp)\hookrightarrow TW(j)$ is injective in cohomology (see Example~\ref{ex.exampleclosedclass})
and the cohomology of its homotopy fiber  $TW(\rho)$ is the
same as
\[\Hom^*_{\K}(\K\gamma,TW(\sK^{<p}))[-2]=TW(\sU,\sK^{<p})[2p-2],\]
where the last equality follows from the fact that $\gamma$ has degree $2p$. Note that,  by Lemma~\ref{lem.criterion},
the homotopy fiber of $\rho\colon TW(j_\perp)\hookrightarrow TW(j)$ is homotopy abelian.
We are now ready to prove the main result of this paper.

\begin{theorem} In the above setup, the composition of the semiregularity map
\[H^1(Z,N_{Z|X})\xrightarrow{\pi}H^{p+1}(X,\Omega_X^{p-1})\]
with the natural map
\[H^{p+1}(X,\Omega_X^{p-1})\to
\H^{2p}(X,\Omega^0_X\mapor{d}\cdots\mapor{d}\Omega^{p-1}_X)\]
annihilates every obstruction to embedded deformations of $Z$ in $X$.
If the Hodge-de Rham spectral sequence of $X$ degenerates at level $E_1$ (e.g. if $X$ is smooth proper \cite{DI}), then  $\pi$ also annihilates every obstruction.
\end{theorem}

\begin{proof}
The Example~\ref{ex.quadrato.morfismo.infinito} applied to the commutative square \eqref{equ.quadrato} gives an $L_{\infty}$ morphism of homotopy fibers
\[ \alpha_{\infty}\colon TW(\chi)\dashrightarrow TW(\rho)\]
and therefore a natural transformation of the associated deformation functors
\[ \alpha_\infty\colon \Def_{TW(\chi)}\to  \Def_{TW(\rho)}.\]
By Theorem~\ref{thm.homotopyfibercontrolligHilb} the functor
$\Def_{TW(\chi)}$ is isomorphic to the functor $\Hilb_{Z|X}$ of embedded deformations; then,
the obstruction map associated with $\alpha_{\infty}$ is
\[ \alpha_{\infty}\colon H^2(TW(\chi))=H^1(Z,N_{Z|X})\to H^2(TW(\rho))=H^{2p}(TW(\sU,\sK^{<p}))\]
and it sends obstructions to embedded deformations to obstructions of the functor $\Def_{TW(\rho)}$. Since $TW(\rho)$ is homotopy abelian, the functor $\Def_{TW(\rho)}$ is unobstructed and this proves that the above map annihilates every obstruction.

Next, in view of the surjective quasi-isomorphisms of quasi-coherent sheaves
$\sK^j\to \sH^p_{Z}(X,\Omega^j_X)$, given by Lemma \ref{lem. quasi iso K cohom locale},
the cohomology of
$TW(\sU,\sK^{<p})$ is isomorphic to the hypercohomology of the complex
\[ TW(\sU,\sH^p_{Z}(\Omega^0_X)\mapor{d}\cdots\mapor{d}\sH^p_{Z}(\Omega^{p-1}_X))\]
and therefore it is also isomorphic to the hypercohomology over $U$ of the complex of sheaves
\[ \sH^p_{Z}(\Omega^0_X)\mapor{d}\cdots\mapor{d}\sH^p_{Z}(\Omega^{p-1}_X).\]
Since $Z$ is a locally complete intersection, the spectral sequence of local cohomology degenerates. Thus,
\[ \H^*(U,\sH^p_{Z}(U,\Omega^0_X)\mapor{d}\cdots\mapor{d}\sH^p_{Z}(U,\Omega^{p-1}_X))=
\H_Z^*(U,\Omega^0_X\mapor{d}\cdots\mapor{d}\Omega^{p-1}_X)\]
and, therefore, the map
\[ \alpha_{\infty}\colon H^1(Z,N_{Z|X})\to H^2(TW(\rho))=
\H_Z^{2p}(U,\Omega^0_X\mapor{d}\cdots\mapor{d}\Omega^{p-1}_X)\]
annihilates the obstructions to embedded deformations.
By general results about local cohomology of finite complexes of sheaves,
we have a natural map
\[ \H_Z^{2p}(U,\Omega^0_X\mapor{d}\cdots\mapor{d}\Omega^{p-1}_X)=
\H_Z^{2p}(X,\Omega^0_X\mapor{d}\cdots\mapor{d}\Omega^{p-1}_X)\to
\H^{2p}(X,\Omega^0_X\mapor{d}\cdots\mapor{d}\Omega^{p-1}_X).\]
In order to conclude the proof, it is now sufficient to prove the commutativity of
the diagram
\[ \xymatrix{H^1(Z,N_{Z|X})\ar[r]^{\alpha_{\infty}\qquad}\ar[d]^{\pi}&
\H_Z^{2p}(U,\Omega^0_X\mapor{d}\cdots\mapor{d}\Omega^{p-1}_X)\ar[d]\\
H^{p+1}(X,\Omega_X^{p-1})\ar[r]&
\H^{2p}(X,\Omega^0_X\mapor{d}\cdots\mapor{d}\Omega^{p-1}_X).}\]
According to the definition of the semiregularity map given in the introduction, it is sufficient to prove that we have a commutative diagram
\[ \xymatrix{H^1(Z,N_{Z|X})\ar[r]^{\contr\{Z\}'}\ar[dr]^{\alpha_{\infty}}&H^{p+1}_Z(U,\Omega_X^{p-1})\ar[d]\\
&\H_Z^{2p}(U,\Omega^0_X\mapor{d}\cdots\mapor{d}\Omega^{p-1}_X).
}\]
Next, the composition of the
linear part of the $L_{\infty}$ morphism
$\alpha_{\infty}\colon TW(\chi)\dashrightarrow TW(\rho)$ with the quasi-isomorphism
$TW(\rho)\to TW(\sU,\sK^{<p})[2p-2]$ described in Remark~\ref{rem.quasiisoTWcono} is induced by the internal product with the canonical cycle class $\{U^0\}'$ and the conclusion follows from the fact that
$\{U^0\}'$ maps onto $\{Z\}'$ via the morphism $\sK^p\to \sH_Z^p(U,\Omega^p_X)$.
\end{proof}

\begin{remark}\label{rem.locallytrivialdeformations}
An essentially equivalent proof, but technically easier,
shows that the semiregularity map annihilates
obstructions to embedded locally trivial deformations of a local complete
intersection subvariety $Z$, without the set-up assumption of the introduction.
To this end it is sufficient to take $(E,\sS)=(X,\Oh_X)$, $\sM=\Theta_X$ the tangent sheaf of $X$ and
$\sM_{\perp}=\Theta_X(-\log(Z))$
the subsheaf of derivations of $\Oh_X$ preserving the ideal of $Z$,
see \cite{donarendiconti}.
\end{remark}

\end{document}